\documentclass[12pt,a4paper, oneside, bold,secthm,seceqn,amsthm,ussrhead,reqno]{article}
\usepackage[utf8]{inputenc}
\usepackage[english]{babel}
\usepackage[symbol]{footmisc}
\usepackage{amssymb,amsmath,amsthm,amsfonts,xcolor,enumerate,hyperref, comment,longtable, cleveref}
\usepackage{verbatim}
\usepackage{times}
\usepackage{cite}
\usepackage{pdflscape}
\usepackage{ulem}
\usepackage[mathcal]{euscript}
\usepackage{tikz}
\usepackage{hyperref}
 
\usepackage{cancel}
\usepackage{stmaryrd}
 
\usepackage{amsfonts}
\usepackage{amssymb}
\usepackage{times}
\usepackage{xcolor}
\usepackage{tkz-graph}
\usepackage{url}
\usepackage{float}
\usepackage{tasks}
\usepackage{array}

\usepackage{cite}
\usepackage{hyperref}

 \usepackage{fancyhdr} 
\usepackage{amsfonts}
\usepackage{amsmath}
\usepackage{eurosym}
\usepackage{geometry}

\usepackage{caption,booktabs}

\theoremstyle{plain}
\newtheorem{theorem}{Theorem}
\newtheorem{lemma}[theorem]{Lemma}
\newtheorem{proposition}[theorem]{Proposition}

\newtheorem{definition}[theorem]{Definition}
\newtheorem{corollary}[theorem]{Corollary}

\usetikzlibrary{arrows}

\newtheorem*{theoremA1}{Theorem A1}
\newtheorem*{theoremA2}{Theorem A2}

\newtheorem*{theoremG0}{Theorem G0}
\newtheorem*{theoremG1}{Theorem G1}
\newtheorem*{theoremG2}{Theorem G2}

\usepackage{fouriernc}

\begin{document}


 \bigskip

\noindent{\Large
The algebraic and geometric classification of \\
$\delta$-Novikov    algebras}

 \bigskip

\begin{center}

 {\bf
Hani Abdelwahab\footnote{Department of Mathematics, 
 Mansoura University,  Mansoura, Egypt; \ haniamar1985@gmail.com}, 
   Ivan Kaygorodov\footnote{CMA-UBI, University of  Beira Interior, Covilh\~{a}, Portugal; \  Moscow Center for Fundamental and Applied Mathematics, Moscow, Russia; \     kaygorodov.ivan@gmail.com}   \&
   Roman Lubkov\footnote{Department of Mathematics and Computer Science, Saint Petersburg State University, Russia; r.lubkov@spbu.ru, romanlubkov@yandex.ru}  
}

\end{center}

\ 

\noindent {\bf Abstract:}
{\it  
The notion of $\delta$-Novikov algebras was introduced recently as a generalization of  Novikov and bicommutative algebras.  It looks like $\delta$-Novikov algebras have a richer structure than Novikov algebras. 
So, unlike Novikov algebras, they have a $2$-dimensional simple algebra for $\delta=-1.$ 
The present paper is dedicated to the study of $3$-dimensional $\delta$-Novikov algebras for $\delta \notin \big\{0,1\big\}.$
The algebraic and geometric classifications of complex 
$3$-dimensional $\delta$-Novikov algebras  are given.  
As a corollary, we prove that there are no simple $3$-dimensional $\delta$-Novikov algebras.}

 \bigskip 


\bigskip 

 \
 

	 \bigskip

\ 

\


\tableofcontents 
\newpage

\section{The algebraic classification of $\protect\delta $-Novikov algebras}

\subsection{Preliminaries: the algebraic   classification}
All the algebras below will be over $\mathbb C$ and all the linear maps will be $\mathbb C$-linear.
For simplicity, every time we write the multiplication table of an algebra 
the products of basic elements whose values are zero or can be recovered from the commutativity  or from the anticommutativity are omitted.
The notion of a nontrivial algebra means that the multiplication is nonzero.
In this section, we introduce the techniques used to obtain our main results (the techniques are similar to those considered  in \cite{afm}).

\medskip

Recall that an algebra $({\rm A}, \cdot)$ is called a  $\delta$-Novikov algebra
if it satisfies  the identities: 
\begin{longtable}{rcl}
       $\delta  (xy)z-x(yz) $&$ = $&$ \delta (yx)z-y(xz),$\\
        $(xy)z $&$=$&$ (xz)y.$
    \end{longtable}

Let $({\rm A}, \cdot)$ be an algebra. We consider the following two new products
on the underlying vector space ${\rm A}$ defined by 
\begin{center}
$[x,y] :=\frac{1}{2} (x\cdot y-y\cdot x).$
\end{center}
Let us denote  
${\rm A}^-:=({\rm A},
[\cdot,\cdot]).$

\begin{definition}
Let $({\rm A},[\cdot,\cdot])$ be a Lie algebra. 
Let ${\rm Z}^2_\delta({\rm A},{\rm A})$ be the set of all symmetric  bilinear maps $\theta :%
{\rm A}\times {\rm A} \to {\rm A}$ such that 
\begin{longtable}{ccccccccccccccccc}
$\delta  \big([[x,y],z]  $&$+$&$[\theta(x,y),z] $&$+$&$\theta([x,y],z)$&$ +$&$\theta(\theta(x,y),z) \big) $\\
& $-$&$
   \big([x,[y,z]] $&$ +$&$[x,\theta(y,z)] $&$+$&$\theta([x,y],z)$&$ +$&$\theta(\theta(x,y),z) \big)$ &$=$\\
$\delta  \big([[y,x],z] $&$ +$&$[\theta(y,x),z] $&$+$&$\theta([y,x],z) $&$+$&$\theta(\theta(y,x),z) \big)$\\
& $-$&$
   \big([y,[x,z]]  $&$+$&$[y,\theta(x,z)] $&$+$&$\theta([y,x],z) $&$+$&$\theta(\theta(y,x),z) \big);$
     \medskip 

\\
        $[[x,y],z]$&$+$& 
        $[\theta(x,y),z]$&$+$& 
        $\theta([x,y],z)$&$+$& 
        $\theta(\theta(x,y),z)$&$=$ \\   
       && $[[x,z],y]$&$+$& 
        $[\theta(x,z),y]$&$+$& 
        $\theta([x,z],y)$&$+$& 
        $\theta(\theta(x,z),y).$
        
\end{longtable}  
\end{definition}

\noindent  For $\theta \in {\rm Z}^2_\delta(
{\rm A},{\rm A})$ we define on ${\rm A}$ a product $%
*_{\theta} :{\rm A}\times {\rm A}\to {\rm A}$ by 
$
x *_{\theta} y := \theta(x,y).  $

\begin{lemma}
Let $({\rm A},\cdot)$ be a Lie algebra and $\theta \in {\rm Z}^2_\delta({\rm A},{\rm A})$. Then  $( {\rm A},\cdot_{\theta})$ is a
$\delta$-Novikov   algebra, 
where \begin{center}
     $x\cdot_{\theta} y := x \cdot y + x *_{\theta} y.$ 
\end{center}

\end{lemma}

Now, let $({\rm A},\cdot)$ be an algebra and $\mathrm{{Aut}(%
{\rm A})}$ be the automorphism group of ${\rm A}$ with respect to
product $\cdot$. Then $\mathrm{{Aut}({\rm A})}$ acts on ${\rm Z}^2_\delta({\rm A},%
{\rm A})$ by 
\begin{equation*}
(\theta \ast \phi)(x,y) := \phi^{-1}\bigl(\theta \bigl(\phi(x),\phi(y) \bigr)%
\bigr),
\end{equation*}
where $\phi \in \mathrm{{Aut}({\rm A})}$ and $\theta \in {\rm Z}^2({\rm A}%
, {\rm A})$.

\begin{lemma}
Let $({\rm A},\cdot)$ be a Lie algebra and $\theta, \vartheta \in {\rm Z}^2_\delta(%
{\rm A},{\rm A})$. Then the   algebras $({\rm A}, \cdot_{\theta})$ and $({\rm A}, \cdot_{\vartheta})$ are isomorphic if and only if there exists $\phi \in \mathrm{{Aut}({\rm A})}$ satisfying $\theta \ast \phi
=\vartheta $.
\end{lemma}

Hence, we have a procedure to classify the $\delta$-Novikov   algebras associated with a given Lie  algebra $({\rm A},\cdot)$. It consists of three steps:

\begin{enumerate}
\item[{\bf Step} $1.$] Compute ${\rm Z}^2_\delta({\rm A},{\rm A})$.

\item[{\bf Step} $2.$] Find the orbits of $\mathrm{{Aut}({\rm A})}$ on ${\rm Z}^2_\delta(
{\rm A},{\rm A})$.

\item[{\bf Step} $3.$] Choose a representative $\theta$ from each orbit and then
construct the $\delta$-Novikov algebra $({\rm A}, \cdot_{\theta})$.
\end{enumerate}

Let us introduce the following notations. Let $\{e_1,\dots,e_n\}$ be a fixed basis of an algebra $({\rm A},\cdot)$. Define 
$\mathrm{\Lambda }^2({\rm A},\mathbb{C})$ to be the space of all 
  symmetric bilinear forms on ${\rm A}$, that is, \begin{center}
 $\mathrm{\Lambda}^2(%
{\rm A},\mathbb{C}) := 
\langle \Delta_{ij} | 1\leq i\leq j \leq n\rangle$,
\end{center} where $\Delta_{ij}$ is the  symmetric  bilinear form $%
\Delta_{ij}:{\rm A}\times {\rm A} \to \mathbb{C}$ defined by 
\begin{equation*}
\Delta_{ij}( e_l,e_m) :=\left\{ 
\begin{tabular}{rl}
$1,$ & if $(i,j) = (l,m)$ or  $(i,j) = (m,l),$   \\ 
$0,$ & otherwise.%
\end{tabular}
\right.
\end{equation*}
Now, if $\theta \in {\rm Z}^2_\delta({\rm A},{\rm A})$ then $\theta$ can be
uniquely written as \begin{center}
     $\theta (x,y) = \sum_{i=1}^n B_i(x,y)e_i$,
\end{center} where $%
B_1,\dots, B_n$ are  symmetric  bilinear forms on ${\rm A}$. Also,
we may write $\theta = \big(B_{1},\dots ,B_n \big)$. Let $\phi^{-1}\in \mathrm{Aut}(%
{\rm A})$ be given by the matrix $( b_{ij})$. If \begin{center}
    $(\theta \ast
\phi)(x,y) = \sum_{i=1}^n B_i^{\prime }(x,y)e_i$,
\end{center} then $B_i^{\prime }=
\sum_{j=1}^n b_{ij}\phi^t B_j\phi$, whenever $i \in \{1,\dots,n\}$.

\begin{proposition}
\label{3-dim metalbleian}Let $\big( \mathcal{L},\left[ \cdot, \cdot \right] \big) $
be a nontrivial complex metabelian Lie algebra of dimension three. Then $%
\mathcal{L}$ is isomorphic to one of the following Lie algebras:

\begin{longtable}{llll}
$\mathcal{L}_{01}$ & $:$ & $[e_{1},e_{2}] = e_{3}$ \\

$\mathcal{L}_{02}$ & $:$ & $[e_{1},e_{2}] = e_{2}$ & $[e_{1},e_{3}] = e_{2} + e_{3}$ \\

$\mathcal{L}_{03}^{\alpha}$ & $:$ & $[e_{1},e_{2}] = e_{2}$ & $[e_{1},e_{3}] = \alpha e_{3}$ \\

\end{longtable}
\noindent
All listed algebras are non-isomorphic except: $\mathcal{L}_{03}^{\alpha}\cong \mathcal{L}_{03}^{\alpha ^{-1}}$.
\end{proposition}

\begin{proposition}
\label{3-dim assoc}Let $\big({\rm A}, \cdot \big)$ be a nontrivial complex  $3$-dimensional
associative commutative algebra. Then ${\rm A}$ is isomorphic to one of
the following algebras:

\begin{longtable}{llllll}
${\rm A}_{01}$ & $:$ & $e_{1}\cdot e_{1} = e_{1}$ & $e_{2}\cdot e_{2} = e_{2}$ & \\ 
${\rm A}_{02}$ & $:$ & $e_{1}\cdot e_{1} = e_{1}$ & $e_{1}\cdot e_{2} = e_{2}$ & \\ 
${\rm A}_{03}$ & $:$ & $e_{1}\cdot e_{1} = e_{1}$ & & \\ 
${\rm A}_{04}$ & $:$ & $e_{1}\cdot e_{1} = e_{2}$ & & \\ 
${\rm A}_{05}$ & $:$ & $e_{1}\cdot e_{2} = e_{3}$ & & \\ 
${\rm A}_{06}$ & $:$ & $e_{1}\cdot e_{1} = e_{2}$ & $e_{1}\cdot e_{2} = e_{3}$ & \\ 
${\rm A}_{07}$ & $:$ & $e_{1}\cdot e_{1} = e_{1}$ & $e_{2}\cdot e_{2} = e_{2}$ & $e_{3}\cdot e_{3} = e_{3}$ \\ 
${\rm A}_{08}$ & $:$ & $e_{1}\cdot e_{1} = e_{1}$ & $e_{2}\cdot e_{2} = e_{2}$ & $e_{2}\cdot e_{3} = e_{3}$ \\ 
${\rm A}_{09}$ & $:$ & $e_{1}\cdot e_{1} = e_{1}$ & $e_{1}\cdot e_{2} = e_{2}$ & $e_{1}\cdot e_{3} = e_{3}$ \\ 
${\rm A}_{10}$ & $:$ & $e_{1}\cdot e_{1} = e_{1}$ & $e_{1}\cdot e_{2} = e_{2}$ & $e_{1}\cdot e_{3} = e_{3}$ & $e_{2}\cdot e_{2} = e_{3}$ \\ 
${\rm A}_{11}$ & $:$ & $e_{1}\cdot e_{1} = e_{1}$ & $e_{2}\cdot e_{2} = e_{3}$ & \\ 
\end{longtable}
 
\end{proposition}

\newpage 
\subsection{The algebraic classification of anti-Novikov algebras}

\begin{theoremA1}
Let ${\rm N}$\ be a complex $3$-dimensional anti-Novikov algebra.
Then ${\rm N}$ is an associative commutative algebra listed in {\rm Proposition %
\ref{3-dim assoc}} or isomorphic to one of the following algebras\footnote{For receiving similar multiplication tables we have to apply the basis change $e_1:=\frac12e_1$ in algebras ${\rm N}_{04}, {\rm N}_{05}^{\alpha}, {\rm N}_{06}, {\rm N}_{07}, {\rm N}_{08}, {\rm N}_{11}, {\rm N}_{14}$, and ${\rm N}_{16}$.}:

\begin{longtable}{lllllll}
${\rm N}_{01}$ & $:$ & $e_{1}\cdot e_{2}=e_{3}$ & $e_{2}\cdot e_{1}=-e_{3}$ \\

${\rm N}_{02}^{\alpha}$ & $:$ & $e_{1}\cdot e_{1}=e_{3}$ & $e_{1}\cdot e_{2}=e_{3}$ & $e_{2}\cdot e_{1}=-e_{3}$ & $e_{2}\cdot e_{2}=\alpha e_{3}$ \\

${\rm N}_{03}^{\alpha}$ & $:$ & $e_{1}\cdot e_{1}=e_{2}$ & $e_{1}\cdot e_{2}=\big( \alpha +1\big) e_{3}$ & $e_{2}\cdot e_{1}=\big( \alpha -1\big) e_{3}$ \\

${\rm N}_{04}$ & $:$ & $e_{1}\cdot e_{2}=e_{2}$ & $e_{1}\cdot e_{3}=e_{2}+e_{3}$ \\

${\rm N}_{05}^{\alpha}$ & $:$ & $e_{1}\cdot e_{2}=e_{2}$ & $e_{1}\cdot e_{3}=\alpha e_{3}$ \\

${\rm N}_{06}$ & $:$ & $e_{1}\cdot e_{1}=e_{1}$ & $e_{2}\cdot e_{1}=-e_{2}$ & $e_{1}\cdot e_{3}=e_{3}$ & $e_{3}\cdot e_{1}=e_{3}$ \\

${\rm N}_{07}$ & $:$ & $e_{1}\cdot e_{1}=e_{1}$ & $e_{2}\cdot e_{1}=-e_{2}$ & $e_{3}\cdot e_{3}=e_{3}$ \\

${\rm N}_{08}$ & $:$ & $e_{1}\cdot e_{1}=e_{1}$ & $e_{2}\cdot e_{1}=-e_{2}$ \\

${\rm N}_{09}$ & $:$ & $e_{1}\cdot e_{1}=2e_{1}$ & $e_{1}\cdot e_{2}=e_{3}$ & $e_{2}\cdot e_{1}=e_{3}-2e_{2}$ & $e_{1}\cdot e_{3}=2e_{3}$ & $e_{3}\cdot e_{1}=2e_{3}$ \\

${\rm N}_{10}$ & $:$ & $e_{1}\cdot e_{1}=e_{3}$ & $e_{1}\cdot e_{2}=2e_{2}$ \\

${\rm N}_{11}$ & $:$ & $e_{1}\cdot e_{2}=e_{2}$ & $e_{3}\cdot e_{3}=e_{3}$ \\

${\rm N}_{12}$ & $:$ & $e_{1}\cdot e_{2}=2e_{2}$ & $e_{2}\cdot e_{2}=e_{1}$ \\

${\rm N}_{13}$ & $:$ & $e_{1}\cdot e_{2}=2e_{2}$ & $e_{2}\cdot e_{2}=e_{1}$ & $e_{3}\cdot e_{3}=e_{3}$ \\

${\rm N}_{14}$ & $:$ & $e_{1}\cdot e_{2}=e_{2}$ & $e_{2}\cdot e_{2}=e_{3}$ \\

${\rm N}_{15}$ & $:$ & $e_{1}\cdot e_{1}=e_{3}$ & $e_{1}\cdot e_{2}=2e_{2}$ &  $e_{2}\cdot e_{2}=e_{3}$ \\

${\rm N}_{16}$ & $:$ & $e_{1}\cdot e_{1}=e_{1}$ & $e_{2}\cdot e_{1}=-e_{2}$ & $e_{3}\cdot e_{1}=-e_{3}$ \\
\end{longtable}

\noindent
All listed algebras are non-isomorphic except: ${\rm N}_{05}^{\alpha
}\cong {\rm N}_{05}^{\alpha ^{-1}}$.
\end{theoremA1}

\begin{proof}
Let ${\rm N}$\ be a complex $3$-dimensional anti-Novikov algebra.
Then ${\rm N}^{-}$ is metabelian \cite{kv16}. If ${\rm N}^{-}$ has the zero multiplication, then $%
{\rm N}$\ is commutative and associative. Otherwise, by Proposition~\ref%
{3-dim metalbleian}, we may assume ${\rm N}^{-}\in \big\{\mathcal{L}_{01},%
\mathcal{L}_{02},\mathcal{L}_{03}^{\alpha }\big\}$. So we study the following
cases:

\begin{enumerate}[I.]
    \item 
\underline{${\rm N}^{-}=\mathcal{L}_{01}$.} Choose an arbitrary element $%
\theta =\big( B_{1},B_{2},B_{3}\big) \in {\rm Z}_{-1}^{2}\big( 
\mathcal{L}_{01},\mathcal{L}_{01}\big) $. Then $\theta \in \left\{ \eta
_{1},\ldots ,\eta _{4}\right\} $ where%
\begin{longtable}{lcll}
$\eta _{1} $&$=$&$\big( 0,\ 0,\ \alpha _{1}\Delta _{11}+\alpha
_{2}\Delta _{22}+\alpha _{3}\Delta _{12}\big) ,$ \\
$\eta _{2} $&$=$&$\big( 0,\ \alpha _{1}\Delta _{11}, \ \alpha
_{2}\Delta _{11}+\alpha _{3}\Delta _{12}\big) ,$ \\
$\eta _{3} $&$=$&$\big( \alpha _{1}\Delta _{22},\ 0,\ \alpha
_{2}\Delta _{22}+\alpha _{3}\Delta _{12}\big) ,$ \\
$\eta _{4} $&$=$&$\left(
\begin{array}{ll}
\alpha _{1}\Delta _{11}+{\alpha _{2}^{2}}{\alpha^{-1} _{1}}\Delta _{22}+\alpha _{2}\Delta _{12}, &-{\alpha _{1}^{2}}{\alpha^{-1} _{2}}\Delta _{11}-\alpha _{2}\Delta
_{22}-\alpha _{1}\Delta _{12},\\
\multicolumn{2}{r}{\alpha _{3}\Delta
_{11}+\alpha _{4}\Delta _{22}+\frac{\alpha _{3}\alpha
_{2}^{2}+\alpha _{4}\alpha _{1}^{2}}{2\alpha _{1}\alpha _{2}}\Delta
_{12}}
\end{array} \right)_{\alpha _{1}\alpha _{2}\neq 0},$
\end{longtable}%

for some $\alpha _{1},\alpha _{2},\alpha _{3},\alpha _{4}\in \mathbb{C}$.
The automorphism group of $\mathcal{L}_{01}$, $\mathrm{Aut}\big( \mathcal{L}%
_{01}\big) $, consists of the invertible matrices of the following form:%
\begin{equation*}
\phi =%
\begin{pmatrix}
a_{11} & a_{12} & 0 \\ 
a_{21} & a_{22} & 0 \\ 
a_{31} & a_{32} & a_{11}a_{22}-a_{12}a_{21}%
\end{pmatrix}%
.
\end{equation*}

\begin{itemize}
\item $\theta =\eta _{1}$. \ Write $\theta \ast \phi =\big( 0,  \ 0 , \ \beta
_{1}\Delta _{11}+\beta _{2}\Delta _{22}+\beta
_{3}\Delta _{12}\big) $. Then%
\begin{longtable}{lcl}
$\beta _{1} $&$=$&$(a_{11}a_{22}-a_{12}a_{21})^{-1} \big( \alpha
_{1}a_{11}^{2}+2\alpha _{3}a_{11}a_{21}+\alpha _{2}a_{21}^{2}\big) ,$ \\
$\beta _{2} $&$=$&$(a_{11}a_{22}-a_{12}a_{21})^{-1}\big( \alpha
_{1}a_{12}^{2}+2\alpha _{3}a_{12}a_{22}+\alpha _{2}a_{22}^{2}\big) ,$ \\
$\beta _{3} $&$=$&$(a_{11}a_{22}-a_{12}a_{21})^{-1}\big( \alpha
_{1}a_{11}a_{12}+\alpha _{3}a_{11}a_{22}+\alpha _{3}a_{12}a_{21}+\alpha
_{2}a_{21}a_{22}\big) .$
\end{longtable}
Then \[
\begin{pmatrix} 
\beta_{1} & \beta_{3} \\ 
\beta_{3} & \beta_{2}
\end{pmatrix}
= \frac{1}{a_{11}a_{22} - a_{12}a_{21}}
\begin{pmatrix}
a_{11} & a_{21} \\ 
a_{12} & a_{22}
\end{pmatrix}
\begin{pmatrix} 
\alpha_{1} & \alpha_{3} \\ 
\alpha_{3} & \alpha_{2}
\end{pmatrix}
\begin{pmatrix}
a_{11} & a_{12} \\ 
a_{21} & a_{22}
\end{pmatrix}.
\]

  Thus, up to a scaler,
$  
\begin{pmatrix} 
\beta _{1} & \beta _{3} \\ 
\beta _{3} & \beta _{2}%
\end{pmatrix}%
 $ and $  
\begin{pmatrix} 
\alpha _{1} & \alpha _{3} \\ 
\alpha _{3} & \alpha _{2}%
\end{pmatrix}%
  $ are equivalent. Since $ 
\begin{pmatrix} 
\alpha _{1} & \alpha _{3} \\ 
\alpha _{3} & \alpha _{2}%
\end{pmatrix}%
  $ is symmetric, we may assume without any loss of generality that $%
\alpha _{3}=0$. Then we have the following cases:

\begin{itemize}
\item $\big( \alpha _{1},\alpha _{2}\big) =\big( 0,0\big) $. Then we
get the algebra ${\rm N}_{01}$.

\item $\big( \alpha _{1},\alpha _{2}\big) \neq \big( 0,0\big) $. Let
$\phi=\varphi_1$ if $\alpha_1 \neq 0$ or $\phi=\varphi_2$  if $%
\alpha_1 =0$:%
\begin{equation*}
\varphi_1 =\begin{pmatrix} 
1 & 0 & 0 \\ 
0 & \alpha _{1} & 0 \\ 
0 & 0 & \alpha _{1}%
\end{pmatrix}%
, \  
\varphi_2 =\begin{pmatrix}
0 & \alpha _{2} & 0 \\ 
1 & 0 & 0 \\ 
0 & 0 & -\alpha _{2}%
\end{pmatrix}.
\end{equation*}%
Then $\theta \ast \phi =\big( 0,\ 0,\ \Delta _{11}+\beta
_{2}\Delta _{22}\big) $. Hence we get the algebras ${\rm N}%
_{02}^{\alpha }$. Moreover,   ${\rm N}_{02}^{\alpha } \cong
{\rm N}_{02}^{\beta }$  if and only if $\alpha =\beta $.
\end{itemize}

\item $\theta =\eta _{2}$. Without any loss of generality, we may assume $%
\alpha _{1}\neq 0$ since otherwise we are back in the case $\theta =\eta
_{1} $. Let $\phi $ be the following automorphism: 
\begin{equation*}
\phi =%
\begin{pmatrix}
1 & 0 & 0 \\ 
0 & \alpha _{1} & 0 \\ 
0 & \alpha _{2} & \alpha _{1}%
\end{pmatrix}%
.
\end{equation*}%
Then $\theta \ast \phi =\big( 0,  \ \Delta _{11},\ \alpha
_{3}\Delta _{12}\big) $. Hence we get the algebras ${\rm N}%
_{03}^{\alpha }$.$\allowbreak $ Furthermore,   ${\rm N}%
_{03}^{\alpha } \cong {\rm N}_{03}^{\beta }$   if and only
if $\alpha =\beta $.

\item $\theta =\eta _{3}$. Without any loss of generality, we may assume $%
\alpha _{1}\neq 0$ since otherwise we are back in the case $\theta =\eta
_{1} $. Let $\phi $ be the following automorphism: 
\begin{equation*}
\phi =%
\begin{pmatrix}
0 & \alpha _{1} & 0 \\ 
1 & 0 & 0 \\ 
0 & \alpha _{2} & \alpha _{1}%
\end{pmatrix}%
.
\end{equation*}%
Then $\theta \ast \phi =\big( 0, \ \Delta _{11},\ \alpha
_{3}\Delta _{12}\big) $. So we get the algebras ${\rm N}%
_{03}^{\alpha }$.

\item $\theta =\eta _{4}$. Let $\phi $ be the following automorphism:%
\begin{equation*}
\phi =%
\begin{pmatrix}
1 & \alpha _{1} & 0 \\ 
0 & -{\alpha _{1}^{2}}{\alpha^{-1} _{2}} & 0 \\ 
0 & \alpha _{3} & -{\alpha _{1}^{2}}{\alpha^{-1} _{2}}%
\end{pmatrix}%
.
\end{equation*}%
Then $\theta \ast \phi =\big( 0,\ \Delta _{11},\ \frac{\alpha_{1}^{2}\alpha _{4}-\alpha _{2}^{2}\alpha _{3}}{2\alpha _{1}\alpha _{2}}%
\Delta _{12}\big) $. So we get the algebras ${\rm N}%
_{03}^{\alpha }$.
\end{itemize}

\item 
\underline{${\rm N}^{-}=\mathcal{L}_{02}$.} Choose an arbitrary element $%
\theta =\big( B_{1},B_{2},B_{3}\big) \in {\rm Z}_{-1 }^{2}\big( \mathcal{L}%
_{02},\mathcal{L}_{02}\big) $. Then 
\begin{center}
    $\theta =\big( 0, \ 
\alpha_{1}\Delta _{11}+\Delta _{12}+\Delta _{13},\ 
\alpha_{2}\Delta _{11}+\Delta _{13}\big) $ for some $\alpha
_{1},\alpha _{2}\in \mathbb{C}$.
\end{center} The automorphism group of $\mathcal{L}_{02}$%
, $\mathrm{Aut}\big( \mathcal{L}_{02}\big) $, consists of the invertible
matrices of the following form:%
\begin{equation*}
\phi =%
\begin{pmatrix}
1 & 0 & 0 \\ 
a_{21} & a_{22} & a_{23} \\ 
a_{31} & 0 & a_{22}%
\end{pmatrix}%
.
\end{equation*}%
Now, we choose $\phi $ \ to be the following automorphism:%
\begin{equation*}
\phi =%
\begin{pmatrix}
1 & 0 & 0 \\ 
\frac{\alpha _{2}-\alpha _{1}}{2} & 1 & 0 \\ 
-\frac{\alpha _{2}}{2} & 0 & 1%
\end{pmatrix}%
.
\end{equation*}%
Then $\theta \ast \phi =\big( 0, \ \Delta _{12}+\Delta_{13},\ \Delta _{13}\big) $. So we get the algebras ${\rm N}%
_{04}$.

\item 
\underline{${\rm N}^{-}=\mathcal{L}_{03}^{-1}$.} Choose an
arbitrary element $\theta =\big( B_{1},B_{2},B_{3}\big) \in {\rm Z}_{-1}^{2}\big( \mathcal{L}_{03}^{-1},\mathcal{L}_{03}^{-1}\big) $. Then 
\begin{center}
    $\theta =( 0,\ \alpha _{1}\Delta_{11}+\Delta _{12},\ \alpha _{2}\Delta _{11}-\Delta
_{13}) $ for some $\alpha _{1},\alpha _{2}\in \mathbb{C}$.
\end{center} The automorphism
group of $\mathcal{L}_{03}^{-1}$, $\mathrm{Aut}\big( \mathcal{L}%
_{03}^{-1}\big) $, consists of the invertible matrices of the
following form:%
\begin{equation*}
\begin{pmatrix}
1 & 0 & 0 \\ 
a_{21} & a_{22} & 0 \\ 
a_{31} & 0 & a_{33}%
\end{pmatrix}%
, \ 
\begin{pmatrix}
-1 & 0 & 0 \\ 
a_{21} & 0 & a_{23} \\ 
a_{31} & a_{32} & 0%
\end{pmatrix}%
.
\end{equation*}%
Now, we choose $\phi $ \ to be the following automorphism:%
\begin{equation*}
\phi =%
\begin{pmatrix}
1 & 0 & 0 \\ 
-\frac{\alpha _{1}}{2} & 1 & 0 \\ 
\frac{\alpha _{2}}{2} & 0 & 1%
\end{pmatrix}%
.
\end{equation*}%
Then $\theta \ast \phi =\big( 0,\ \Delta _{12},\ -\Delta
_{13}\big) $. Hence we get the algebra ${\rm N}_{05}^{-1}$.

\item 
\underline{${\rm N}^{-}=\mathcal{L}_{03}^{0}$.} Choose an
arbitrary element $\theta =\big( B_{1},B_{2},B_{3}\big) \in {\rm Z}_{-1}^{2}\big( \mathcal{L}_{03}^{0},\mathcal{L}_{03}^{0}\big) $. Then $\theta \in \big\{ \eta _{1},\ldots ,\eta _{13}\big\},$
where

 \begin{longtable}{lcl}
\(\eta_1\) & $=$ & \(\big( 2\Delta_{11},\  \alpha_1 \Delta_{11} - \Delta_{12},\  2\Delta_{13} + \alpha_2 \Delta_{33} \big)\), \\ 
\(\eta_2\) & $=$ & \(\big( 2\Delta_{11},\ \alpha_1 \Delta_{11} - \Delta_{12},\ \alpha_2 \Delta_{33} \big)\), \\ 
\(\eta_3\) & $=$ & \(\big( 2\Delta_{11},\ \alpha_1 \Delta_{11} - \Delta_{12},\ \alpha_2 \Delta_{11} + \alpha_3 \Delta_{13} + {\alpha_3(\alpha_3 - 2)}{\alpha^{-1}_2} \Delta_{33} \big)\), \\ 
\(\eta_4\) & $=$ & \(\big( 2\Delta_{11},\ \alpha_1 \Delta_{11} - \Delta_{12},\ \alpha_2 \Delta_{11} + \alpha_3 \Delta_{12} + 2\Delta_{13} \big)\), \\ 
\(\eta_5\) & $=$ & \(\big( 0,  \ \alpha_1 \Delta_{11} + \Delta_{12}, \ \alpha_2 \Delta_{11} + {\alpha_3^2}{\alpha^{-1}_2} \Delta_{33} + \alpha_3 \Delta_{13} \big)\), \\ 
\(\eta_6\) & $=$ & \(\big( \alpha_1 \Delta_{22}, \ \Delta_{12},\ \alpha_2 \Delta_{33} \big)\), \\ 
\(\eta_7\) & $=$ & \(\big( 0,\ \alpha_1 \Delta_{11} + \Delta_{12},\ \alpha_2 \Delta_{33} \big)\), \\ 
\(\eta_8\) & $=$ & \(\big( 0,\ \Delta_{12}, \ \alpha_1 \Delta_{11} + \alpha_2 \Delta_{22} \big)\), \\ 
\(\eta_9\) & $=$ & \(\big( 0,\ \alpha_1 \Delta_{11} + \Delta_{12},\  \alpha_2 \Delta_{11} \big)\), \\ 
\(\eta_{10}\) & $=$ & \(\big( 0,\   {2\alpha_3}{\alpha^{-1}_2} \Delta_{11} + \Delta_{12},\ \alpha_1 \Delta_{11} + \alpha_2 \Delta_{22} + \alpha_3 \Delta_{12} \big)\), \\ 
\(\eta_{11}\) & $=$ & \(\big( \alpha_1 \Delta_{22}, \ \Delta_{12}, \ -{\alpha_2 \alpha_3}{\alpha^{-1}_1} \Delta_{11} + \alpha_3 \Delta_{22} - {\alpha_1 \alpha_2}{\alpha^{-1}_3} \Delta_{33} + \alpha_2 \Delta_{13} \big)\), \\ 
\(\eta_{12}\) & $=$ & \(\big( (1 - \alpha_1) \Delta_{11} + \frac{\alpha_2^2}{1 - \alpha_1} \Delta_{22} + \alpha_2 \Delta_{12},\  \frac{1 - \alpha_1^2}{\alpha_2} \Delta_{11} - \alpha_2 \Delta_{22} + \alpha_1 \Delta_{12}, \ \alpha_3 \Delta_{33} \big)\), \\ 
\(\eta_{13}\) & $=$ & 
\(\left( 
\begin{array}{ll}
(1 - \alpha_1) \Delta_{11} + \frac{\alpha_2^2}{1 - \alpha_1} \Delta_{22} + \alpha_2 \Delta_{12}, & 
{(1 - \alpha_1^2)}{\alpha^{-1}_2} \Delta_{11} - \alpha_2 \Delta_{22} + \alpha_1 \Delta_{12}, \\ 
\multicolumn{2}{r}{{(\alpha_3 - \alpha_1 \alpha_3 - \alpha_3 \alpha_4)}{\alpha^{-1}_2} \Delta_{11} + \frac{\alpha_2 \alpha_3}{1 - \alpha_1} \Delta_{22} - {\alpha_2 \alpha_4}{\alpha^{-1}_3} \Delta_{33} + \alpha_3 \Delta_{12} + \alpha_4 \Delta_{13} 
}
\end{array}
\right)\), \\ 
\end{longtable}
for some $\alpha _{1},\alpha _{2},\alpha _{3},\alpha _{4}\in 
\mathbb{C}
.$ The automorphism group of $\mathcal{L}_{03}^{0}$, $\mathrm{Aut}%
\big( \mathcal{L}_{03}^{0}\big) $, consists of the invertible
matrices of the following form:%
\begin{equation*}
\phi =%
\begin{pmatrix}
1 & 0 & 0 \\ 
a_{21} & a_{22} & 0 \\ 
a_{31} & 0 & a_{33}%
\end{pmatrix}%
.
\end{equation*}

\begin{itemize}
\item $\theta =\eta _{1}$. Let $\phi=\varphi_1$   if $\alpha _{2}=0$ or 
$\phi=\varphi_2$  if $\alpha _{2}\neq 0$.%
\begin{equation*}
\varphi_1=\begin{pmatrix}
1 & 0 & 0 \\ 
\frac{\alpha _{1}}{4} & 1 & 0 \\ 
0 & 0 & 1%
\end{pmatrix}%
, \ 
\varphi_2=\begin{pmatrix}
1 & 0 & 0 \\ 
\frac{\alpha _{1}}{4} & 1 & 0 \\ 
-\frac{2}{\alpha _{2}} & 0 & \frac{1}{\alpha _{2}}%
\end{pmatrix}%
.
\end{equation*}%
Then $\theta \ast \phi \in \big\{\big( 2\Delta _{11}, \ -\Delta_{12}, \ 2\Delta _{13}\big), \ \big( 2\Delta_{11}, \ -\Delta _{12}, \ \Delta _{33}\big) \big\}$. So we get the
algebras ${\rm N}_{06}$ and ${\rm N}_{07}$.

\item $\theta =\eta _{2}$. Let $\phi=\varphi_1 $  if $\alpha _{2}\neq 0$ or  $\phi=\varphi_2$  if $\alpha _{2}=0$.%
\begin{equation*}
\varphi_1=\begin{pmatrix}
1 & 0 & 0 \\ 
\frac{\alpha _{1}}{4} & 1 & 0 \\ 
0 & 0 & \frac{1}{\alpha _{2}}%
\end{pmatrix}%
, \ 
\varphi_2=\begin{pmatrix}
1 & 0 & 0 \\ 
\frac{\alpha _{1}}{4} & 1 & 0 \\ 
0 & 0 & 1%
\end{pmatrix}%
.
\end{equation*}%
Then $\theta \ast \phi \in \big\{\big( 2\Delta _{11},\ -\Delta_{12},\ \Delta _{33}\big) ,\ 
\big( 2\Delta_{11}, \ -\Delta _{12},\ 0\big) \big\}$. Hence we get the algebras     ${\rm N}_{07}$  and ${\rm N}_{08}$.

\item $\theta =\eta _{3}$. Let $\phi=\varphi_1 $   if $\alpha _{3}=2$ or $\phi=\varphi_2$   if $\alpha _{3}=0$, or $\phi=\varphi_3$  if $\alpha _{3}\big( \alpha _{3}-2\big) \neq 0$:%
\begin{equation*}
\varphi_1=\begin{pmatrix}
1 & 0 & 0 \\ 
\frac{\alpha _{1}}{4} & 1 & 0 \\ 
-\frac{\alpha _{2}}{2} & 0 & 1%
\end{pmatrix}%
, \ 
\varphi_2=\begin{pmatrix}
1 & 0 & 0 \\ 
\frac{\alpha _{1}}{4} & 1 & 0 \\ 
\frac{\alpha _{2}}{2} & 0 & 1%
\end{pmatrix}%
, \ 
\varphi_3=\begin{pmatrix}
1 & 0 & 0 \\ 
\frac{\alpha _{1}}{4} & 1 & 0 \\ 
-\frac{\alpha _{2}}{\alpha _{3}-2} & 0 & \frac{\alpha _{2}}{\alpha
_{3}\big( \alpha _{3}-2\big) }%
\end{pmatrix}%
.
\end{equation*}%
Then \begin{center}$\theta \ast \phi \in \big\{\big( 2\Delta _{11}, \ -\Delta_{12},\ 2\Delta _{13}\big) ,\ 
\big( 2\Delta_{11},\ -\Delta _{12},\ 0\big), \ 
\big( 2\Delta_{11}, \ -\Delta _{12}, \ \Delta _{33}\big) \big\}$.\end{center} So we get the
algebras ${\rm N}_{06},$ \ ${\rm N}_{08}$, and ${\rm N}_{07}$.

\item $\theta =\eta _{4}$. Let $\phi=\varphi_1$  if $\alpha _{3}=0$ or $\phi=\varphi_2$  if $\alpha _{3}\neq 0$:%
\begin{equation*}
\varphi_1=\begin{pmatrix}
1 & 0 & 0 \\ 
\frac{\alpha _{1}}{4} & 1 & 0 \\ 
-\frac{\alpha_{2}}{2} & 0 & 1%
\end{pmatrix}%
, \ 
\varphi_2=\begin{pmatrix}
1 & 0 & 0 \\ 
\frac{\alpha _{1}}{4} & 1 & 0 \\ 
-\frac{\alpha _{2}}{2}-\frac{\alpha _{1}\alpha _{3}}{4} & 0 & \alpha _{3}%
\end{pmatrix}%
.
\end{equation*}%
Then \begin{center}$\theta \ast \phi \in \big\{\big( 2\Delta _{11}, \ -\Delta_{12},\ 2\Delta _{13}\big), \ 
\big( 2\Delta_{11}, \ -\Delta _{12},\ \Delta _{12}+2\Delta_{13}\big) \big\}$. 
 \end{center}
Thus we obtain the algebras ${\rm N}_{06}$ and ${\rm N}_{09}$.

\item $\theta =\eta _{5}$. Let $\phi=\varphi_1$  if $\alpha _{3}=0$ or $\phi=\varphi_2$ if $\alpha _{3}\neq 0$.%
\begin{equation*}
\varphi_1=\begin{pmatrix}
1 & 0 & 0 \\ 
-\frac{\alpha _{1}}{2} & 1 & 0 \\ 
0 & 0 & \alpha _{2}%
\end{pmatrix}%
, \ 
\varphi_2=\begin{pmatrix}
1 & 0 & 0 \\ 
-\frac{\alpha _{1}}{2} & 1 & 0 \\ 
-\frac{\alpha _{2}}{\alpha _{3}} & 0 & \frac{\alpha _{2}}{\alpha _{3}^{2}}%
\end{pmatrix}%
.
\end{equation*}%
Then $\theta \ast \phi \in \big\{\big( 0, \ \Delta _{12},\ \Delta _{11}\big), \ 
\big( 0,\ \Delta _{12},\ \Delta _{33}\big) \big\}.$
 We obtain the algebras ${\rm N}_{10}$ and ${\rm N}_{11}$.

\item $\theta =\eta _{6}$. If $\alpha _{1}=\alpha _{2}=0$, we get the
algebra ${\rm N}_{05}^{0}$. Assume now that $\big( \alpha
_{1},\alpha _{2}\big) \neq \big( 0,0\big) $. Let $\phi=\varphi_1$  if $\alpha _{1}\neq 0$ and $\alpha _{2}=0$, or $\phi=\varphi_2$ if $\alpha _{1}=0$ and $\alpha _{2}\neq 0$, or 
$\phi=\varphi_3$  if $\alpha
_{1}\alpha _{2}\neq 0$:%
\begin{equation*}
\varphi_1=\begin{pmatrix}
1 & 0 & 0 \\ 
0 &  \alpha^{-\frac{1}{2}}_{1} & 0 \\ 
0 & 0 & 1%
\end{pmatrix}%
, \ 
\varphi_2=\begin{pmatrix}
1 & 0 & 0 \\ 
0 & 1 & 0 \\ 
0 & 0 & {\alpha^{-1}_{2}}%
\end{pmatrix}%
,  \ 
\varphi_3=\begin{pmatrix}
1 & 0 & 0 \\ 
0 &  \alpha^{-\frac{1}{2}}_{1}  & 0 \\ 
0 & 0 &  {\alpha^{-1}_{2}}%
\end{pmatrix}%
.
\end{equation*}%
Then \begin{center} $\theta \ast \phi \in \big\{\big( \Delta _{22},\ \Delta_{12},\ 0\big), \ 
\big( 0,\ \Delta _{12},\ \Delta _{33}\big), \ \big( \Delta _{22},\ \Delta _{12},\ \Delta_{33}\big) \big\}$.
 \end{center} Thus we obtain the algebras ${\rm N}_{12},$ \ ${\rm N}_{11}$ and ${\rm N}_{13}$.

\item $\theta =\eta _{7}$. Let $\phi $ be the following automorphism: 
\begin{equation*}
\phi =%
\begin{pmatrix}
1 & 0 & 0 \\ 
-\frac{\alpha _{1}}{2} & 1 & 0 \\ 
0 & 0 & 1%
\end{pmatrix}%
.
\end{equation*}%
Then $\theta \ast \phi =\big( 0,\ \Delta _{12},\ \alpha
_{2}\Delta _{33}\big) $. So we are back in the case $\theta =\eta
_{6}$.

\item $\theta =\eta _{8}$. We may assume $\big( \alpha _{1},\alpha
_{2}\big) \neq \big( 0,0\big) $ since otherwise we are back in the case 
$\theta =\eta _{6}$. 
Let $\phi=\varphi_1 $  if $\alpha _{1}\neq 0$ and $\alpha _{2}=0$, or 
$\phi=\varphi_2$   if $\alpha _{1}=0$ and $\alpha _{2}\neq 0$, or 
$\phi=\varphi_3$  if $\alpha _{1}\alpha _{2}\neq 0$:%
\begin{equation*}
\varphi_1=\begin{pmatrix}
1 & 0 & 0 \\ 
0 & 1 & 0 \\ 
0 & 0 & \alpha _{1}%
\end{pmatrix}%
, \ 
\varphi_2=\begin{pmatrix}
1 & 0 & 0 \\ 
0 & 1 & 0 \\ 
0 & 0 & \alpha _{2}%
\end{pmatrix}%
, \ 
\varphi_3=\begin{pmatrix}
1 & 0 & 0 \\ 
0 & \sqrt{{\alpha _{1}}{\alpha^{-1}_{2}}} & 0 \\ 
0 & 0 & \alpha _{1}%
\end{pmatrix}%
.
\end{equation*}%
Then \begin{center}$\theta \ast \phi \in \big\{\big( 0, \ \Delta _{12}, \ \Delta_{11}\big), \ 
\big( 0, \ \Delta _{12}, \ \Delta _{22}\big),
\big( 0, \ \Delta _{12}, \ \Delta _{11}+\Delta_{22}\big) \big\}$. 
 \end{center}
Thus we obtain the algebras ${\rm N}_{10},$ \ ${\rm N}_{14}$, and ${\rm N}_{15}$.

\item $\theta =\eta _{9}$. We may assume $\alpha _{2}\neq 0$ since otherwise
we are back in the case $\theta =\eta _{7}$. Let $\phi $ be the following
automorphism:%
\begin{equation*}
\phi =%
\begin{pmatrix}
1 & 0 & 0 \\ 
-\frac{\alpha _{1}}{2} & 1 & 0 \\ 
0 & 0 & \alpha _{2}%
\end{pmatrix}%
.
\end{equation*}%
Then $\theta \ast \phi =\big( 0, \ \Delta _{12},\ \Delta
_{11}\big) $. So we get the algebra ${\rm N}_{10}$.

\item $\theta =\eta _{10}$. Let $\phi $ be the following automorphism:%
\begin{equation*}
\phi =%
\begin{pmatrix}
1 & 0 & 0 \\ 
-{\alpha^{-1} _{2}}{\alpha _{3}} & 1 & 0 \\ 
0 & 0 & \alpha _{2}%
\end{pmatrix}%
.
\end{equation*}%
Then $\theta \ast \phi =\big( 0,\ \Delta _{12},\ ({\alpha_{1}\alpha _{2}-\alpha _{3}^{2}}){\alpha _{2}^{-2}}\Delta
_{11}+\Delta _{22}\big) $. So we have  the case $\theta
=\eta _{8}$.

\item $\theta =\eta _{11}$. Let $\phi $ be the following automorphism:%
\begin{equation*}
\phi =%
\begin{pmatrix}
1 & 0 & 0 \\ 
0 &  {\alpha^{-\frac{1}{2}} _{1}} & 0 \\ 
 {\alpha^{-1} _{1}}{\alpha _{3}} & 0 & 1%
\end{pmatrix}%
.
\end{equation*}%
Then $\theta \ast \phi =\big( \Delta _{22},\ 
\Delta _{12},\ -{\alpha _{1}\alpha _{2}}{\alpha^{-1} _{3}}\Delta _{33}\big) $. So
we are back in the case $\theta =\eta _{6}$.

\item $\theta =\eta _{12}$. Let $\phi $ be the following automorphism:%
\begin{equation*}
\phi =%
\begin{pmatrix}
1 & 0 & 0 \\ 
({\alpha _{1}-1}){\alpha^{-1} _{2}} & {\alpha^{-1} _{2}}\sqrt{1-\alpha _{1}} & 0 \\ 
0 & 0 & 1%
\end{pmatrix}%
.
\end{equation*}%
Then $\theta \ast \phi =\big( \Delta _{22}, \ \Delta_{12}, \ \alpha _{3}\Delta _{33}\big) $. So we are back in the case $%
\theta =\eta _{6}$.

\item $\theta =\eta _{13}$. Let $\phi $ be the following automorphism:%
\begin{equation*}
\phi =%
\begin{pmatrix}
1 & 0 & 0 \\ 
({\alpha _{1}-1}){\alpha^{-1}_{2}} & {\alpha^{-1}_{2}}\sqrt{1-\alpha _{1}} & 0 \\ 
{\alpha^{-1} _{2}}{\alpha _{3}} & 0 & 1%
\end{pmatrix}%
.
\end{equation*}%
Then $\theta \ast \phi =\big( \Delta _{22},\ \Delta _{12}, \ - {\alpha _{2}}{\alpha^{-1} _{3}}\alpha _{4}\Delta _{33}\big) $. So
we are back in the case $\theta =\eta _{6}$.
\end{itemize}

\item 
\underline{${\rm N}^{-}=\mathcal{L}_{03}^{1}$.} Choose an
arbitrary element $\theta =\big( B_{1},B_{2},B_{3}\big) \in {\rm Z}_{-1}^{2}\big( \mathcal{L}_{03}^{1},\mathcal{L}_{03}^{1}\big) $. Then $\theta \in \big\{ \eta _{1},\eta _{2}\big\} $ where%
\begin{longtable}{lcl}
$\eta _{1} $&$=$&$\big( 2\Delta _{11},\ \alpha _{1}\Delta_{11}-\Delta _{12},\ \alpha _{2}\Delta _{11}-\Delta
_{13}\big) ,$ \\
$\eta _{2} $&$=$&$\big( 0,\ \alpha _{1}\Delta _{11}+\Delta_{12},\ \alpha _{2}\Delta _{11}+\Delta _{13}\big) ,$
\end{longtable}%
for some $\alpha _{1},\alpha _{2}\in \mathbb{C}$. The automorphism group of $%
\mathcal{L}_{03}^{1}$, $\mathrm{Aut}\big( \mathcal{L}_{03}^{1}\big) $, consists of the invertible matrices of the following form:%
\begin{equation*}
\phi =%
\begin{pmatrix}
1 & 0 & 0 \\ 
a_{21} & a_{22} & a_{23} \\ 
a_{31} & a_{32} & a_{33}%
\end{pmatrix}%
.
\end{equation*}

\begin{itemize}
\item $\theta =\eta _{1}$. Let $\phi $ be the following automorphism:%
\begin{equation*}
\phi =\allowbreak 
\begin{pmatrix}
1 & 0 & 0 \\ 
\frac{\alpha _{1} }{4}& 1 & 0 \\ 
\frac{\alpha _{2}}{4} & 0 & 1%
\end{pmatrix}%
.
\end{equation*}%
Then $\theta \ast \phi =\big( 2\Delta _{11},-\Delta
_{12},-\Delta _{13}\big) $. Thus we obtain the algebra ${\rm N}_{16}$.

\item $\theta =\eta _{2}$. Let $\phi $ be the following automorphism:%
\begin{equation*}
\phi =\allowbreak 
\begin{pmatrix}
1 & 0 & 0 \\ 
-\frac{1}{2}\alpha _{1} & 1 & 0 \\ 
-\frac{1}{2}\alpha _{2} & 0 & 1%
\end{pmatrix}%
.
\end{equation*}%
Then $\theta \ast \phi =\big( 0,\ \Delta _{12},\ \Delta
_{13}\big) $. Thus we obtain the algebra ${\rm N}_{05}^{1}$.
\end{itemize}

\item 
\underline{${\rm N}^{-}=\mathcal{L}_{03}^{\alpha \notin \{0,\pm 1\}}$.} Choose
an arbitrary element \begin{center}
    $\theta =\big( B_{1},B_{2},B_{3}\big) \in {\rm Z}_{-1}^{2}\big( \mathcal{L}_{03}^{\alpha \notin \{0,\pm 1\}},\mathcal{L}%
_{03}^{\alpha \notin \{0,\pm 1\}}\big) $.
\end{center} Then $\theta = \big(0,\ \alpha_{1}\Delta _{11}+\Delta _{12}, \ \alpha _{2}\Delta_{11}+\alpha \Delta _{13}\big)$ for some $\alpha _{1},\alpha _{2}\in 
\mathbb{C}$. The automorphism group of $\mathcal{L}_{03}^{\alpha \notin \{ 0,\pm
1\}}$, $\mathrm{Aut}\big( \mathcal{L}_{03}^{\alpha \notin \{0,\pm 1\}}\big) $,
consists of the invertible matrices of the following form:%
\begin{equation*}
\phi =%
\begin{pmatrix}
1 & 0 & 0 \\ 
a_{21} & a_{22} & 0 \\ 
a_{31} & 0 & a_{33}%
\end{pmatrix}%
.
\end{equation*}%
Let $\phi $ be the following automorphism:

\begin{equation*}
\phi =%
\begin{pmatrix}
1 & 0 & 0 \\ 
-\frac{\alpha _{1}}{2} & 1 & 0 \\ 
-\frac{\alpha _{2}}{2\alpha } & 0 & 1%
\end{pmatrix}%
.
\end{equation*}%
Then $\theta \ast \phi =\big( 0,\ \Delta _{12},\ \alpha \Delta_{13}\big) $. So we have the algebras ${\rm N}_{05}^{\alpha \notin \{0,\pm 1\}}$.
Moreover, the algebras ${\rm N}_{05}^{\alpha }$ and ${\rm N}%
_{05}^{\beta }$\ are isomorphic if and only if $  \alpha =\beta^{-1}$.

\end{enumerate}
\end{proof}

\subsection{The algebraic classification of $\delta $-Novikov algebras $\big(\delta \notin \big\{0,\pm 1\big\}\big)$}

\begin{theoremA2}
Let $\mathcal{G}$\ be a complex $3$-dimensional $\delta $-Novikov algebra 
with $\delta \notin \big\{ 0,\pm 1\big\}$. Then $\mathcal{G}$ is an associative commutative
algebra listed in {\rm Proposition \ref{3-dim assoc}} or isomorphic to one of the
following algebras\footnote{For receiving similar multiplication tables we have to apply the basis change $e_1:=\frac12e_1$ in algebras $\mathcal{G}_{04},\mathcal{G}_{05}^{\alpha},\mathcal{G}_{06}^{\delta},\mathcal{G}_{07}$ and the basis change $e_1:=-\frac12e_1$ in algebras $\mathcal{G}_{09}^{\delta},\mathcal{G}_{10}^{\delta},\mathcal{G}_{11}^{\delta}$, and $\mathcal{G}_{12}^{\delta}$.}:

\begin{longtable}{lllllll}
$\mathcal{G}_{01}$ & $:$ & $e_{1}\cdot e_{2}=e_{3}$ & $e_{2}\cdot e_{1}=-e_{3}$ \\

$\mathcal{G}_{02}^{\alpha}$ & $:$ & $e_{1}\cdot e_{1}=e_{3}$ & $e_{1}\cdot e_{2}=e_{3}$ & $e_{2}\cdot e_{1}=-e_{3}$ & $e_{2}\cdot e_{2}=\alpha e_{3}$ \\

$\mathcal{G}_{03}^{\alpha}$ & $:$ & $e_{1}\cdot e_{1}=e_{2}$ & $e_{1}\cdot e_{2}=\big( \alpha +1\big) e_{3}$ & $e_{2}\cdot e_{1}=\big( \alpha -1\big) e_{3}$ \\


$\mathcal{G}_{04}$ & $:$ & $e_{1}\cdot e_{2}=e_{2}$ & $e_{1}\cdot e_{3}=e_{2}+e_{3}$ \\


$\mathcal{G}_{05}^{\alpha}$ & $:$ & $e_{1}\cdot e_{2}=e_{2}$ & $e_{1}\cdot e_{3}=\alpha e_{3}$ \\


$\mathcal{G}_{06}^{\delta}$ & $:$ & $e_{1}\cdot e_{2}=e_{2}$  &  $e_{1}\cdot e_{3}=\big( \delta +1\big) e_{3}$ & $e_{2}\cdot e_{2}=e_{3}$\\


$\mathcal{G}_{07}$ & $:$ & $e_{1}\cdot e_{2}=e_{2}$ & $e_{3}\cdot e_{3}=e_{3}$ \\

$\mathcal{G}_{08}$ & $:$ & $e_{1}\cdot e_{1}=e_{3}$ & $e_{1}\cdot e_{2}=2e_{2}$ \\


$\mathcal{G}_{09}^{\delta}$ & $:$ & $e_{1}\cdot e_{1}=\delta e_{1}$ & $e_{2}\cdot e_{1}=e_{2}$ \\


$\mathcal{G}_{10}^{\delta}$ & $:$ & $e_{1}\cdot e_{1}=\delta e_{1}$ & $e_{2}\cdot e_{1}=e_{2}$ & $e_{1}\cdot e_{3}=\delta e_{3}$ & $e_{3}\cdot e_{1}=\delta e_{3}$ \\


$\mathcal{G}_{11}^{\delta}$ & $:$ & $e_{1}\cdot e_{1}=\delta e_{1}$ & $e_{2}\cdot e_{1}=e_{2}$ & $e_{3}\cdot e_{3}=e_{3}$ \\


$\mathcal{G}_{12}^{\delta}$ & $:$ & $e_{1}\cdot e_{1}=\delta e_{1}$ & $e_{2}\cdot e_{1}=e_{2}$ & $e_{3}\cdot e_{1}=e_{3}$ \\
\end{longtable}

\noindent
All listed algebras are non-isomorphic except: $\mathcal{G}_{05}^{\alpha
}\cong \mathcal{G}_{05}^{\alpha ^{-1}}$.
\end{theoremA2}

\begin{proof}
Let $\mathcal{G}$\ be a complex $3$-dimensional $\delta$-Novikov algebra with $\delta \notin \{ 0,\pm 1\}$.
Due to \cite[Lemma 31]{kv16} we have that $\mathcal{G}^{-}$ is metabelian. If $\mathcal{G}^{-}$ has the zero multiplication, then 
 $
\mathcal{G}$\ is commutative and associative. Otherwise, by Proposition \ref%
{3-dim metalbleian}, we may assume $\mathcal{G}^{-}\in \{\mathcal{L}_{01},%
\mathcal{L}_{02},\mathcal{L}_{03}^{\alpha }\}$. So we study the following
cases:

    \begin{enumerate}[I.]
\item \underline{$\mathcal{G}^{-}=\mathcal{L}_{01}$.} Then ${\rm Z}_{\delta }^{2}\big( 
\mathcal{L}_{01},\mathcal{L}_{01}\big) ={\rm Z}_{-1}^{2}\big( \mathcal{L%
}_{01},\mathcal{L}_{01}\big) $. So we obtain the algebras $\mathcal{G}_{01},$ \ $\mathcal{G}_{02},$ and  $\mathcal{G}_{03}^{\alpha }$.

\item 
\underline{$\mathcal{G}^{-}=\mathcal{L}_{02}$.} Then ${\rm Z}_{\delta }^{2}\big( 
\mathcal{L}_{02},\mathcal{L}_{02}\big) ={\rm Z}_{-1}^{2}\big( \mathcal{L%
}_{02},\mathcal{L}_{02}\big) $. So we obtain the algebra $\mathcal{G}_{04}$%
.

\item 
\underline{$\mathcal{G}^{-}=\mathcal{L}_{03}^{-1}$.} Assume first
that $\delta \neq -2$. Then ${\rm Z}_{\delta }^{2}\big( \mathcal{L}_{03}^{-1},\mathcal{L}_{03}^{-1}\big) ={\rm Z}_{-1}^{2}\big( \mathcal{%
L}_{03}^{-1},\mathcal{L}_{03}^{-1}\big) $. Hence we get
the algebra $\mathcal{G}_{05}^{-1}$. 

\item 
\underline{$\mathcal{G}^{-}=\mathcal{L}_{03}^{-1}$.} 
Assume now that $\delta =-2$.
Choose an arbitrary element $\theta =\big( B_{1},B_{2},B_{3}\big) \in
{\rm Z}_{-2}^{2}\big( \mathcal{L}_{01},\mathcal{L}_{01}\big) $. Then $%
\theta \in \left\{ \eta _{1},\ldots ,\eta _{5} \right\} $ where%
\begin{longtable}{lcl}
$\eta _{1} $&$=$&$\big( 0,\Delta _{12},\alpha _{1}\Delta
_{11}+\alpha _{2}\Delta _{22}-\Delta _{13}\big) ,$ \\
$\eta _{2} $&$=$&$\big( 0,\alpha _{1}\Delta _{11}+\alpha
_{2}\Delta _{33}+\Delta _{12},-\Delta _{13}\big) ,$
\\
$\eta _{3} $&$=$&$\big( 0,\alpha _{1}\Delta _{11}+\Delta
_{12},\alpha _{2}\Delta _{11}-\Delta _{13}\big),$ 
\\
$\eta _{4} $&$=$&$\big( 0,{2{\alpha^{-1} _{2}}\alpha _{3}} \Delta
_{11}+\Delta _{12},\alpha _{1}\Delta _{11}+\alpha
_{2}\Delta _{22}+\alpha _{3}\Delta _{12}-\Delta
_{13}\big) ,$ \\
$\eta _{5} $&$=$&$\big( 0,\alpha _{1}\Delta _{11}+\alpha
_{2}\Delta _{33}+\Delta _{12}+\alpha _{3}\Delta
_{13},-{2{\alpha^{-1} _{2}}\alpha _{3}}\Delta _{11}-\Delta
_{13}\big) ,$
\end{longtable}%
for some $\alpha _{1},\alpha _{2},\alpha _{3},\alpha _{4}\in \mathbb{C}$.

\begin{itemize}
\item $\theta =\eta _{1}$. Let $\phi=\varphi_1$  if $\alpha _{2}=0,$ or 
$\phi=\varphi_2$   if $\alpha _{2}\neq 0$:%
\begin{equation*}
\varphi_1=\begin{pmatrix}
1 & 0 & 0 \\ 
0 & 1 & 0 \\ 
\frac{\alpha _{1}}{2} & 0 & 1%
\end{pmatrix}%
, \ 
\varphi_2=\begin{pmatrix}
1 & 0 & 0 \\ 
0 & 1 & 0 \\ 
\frac{\alpha _{1}}{2} & 0 & \alpha _{2}%
\end{pmatrix}%
.
\end{equation*}%
Then $\theta \ast \phi \in \big\{\big( 0,\ \Delta _{12},\ -\Delta
_{13}\big) , \ \big( 0,\ \Delta _{12},\ \Delta
_{22}-\Delta _{13}\big) \big\}$. We get  $\mathcal{G%
}_{05}^{-1}$ and   $\mathcal{G}_{06}^{-2}$.

\item $\theta =\eta _{2}$. Let $\phi=\varphi_1$   if $\alpha _{2}=0,$ or 
$\phi=\varphi_2$  if $\alpha _{2}\neq 0$:%
\begin{equation*}
\varphi_1=\begin{pmatrix}
-1 & 0 & 0 \\ 
\frac{\alpha _{1}}{2} & 0 & 1 \\ 
0 & 1 & 0%
\end{pmatrix}%
, \ 
\varphi_2=\begin{pmatrix}
-1 & 0 & 0 \\ 
\frac{\alpha _{1}}{2} & 0 & \alpha _{2} \\ 
0 & 1 & 0%
\end{pmatrix}%
.
\end{equation*}%
Then $\theta \ast \phi \in \big\{\big( 0,\ \Delta _{12},\ -\Delta
_{13}\big),\ 
\big( 0,\ \Delta _{12},\ \Delta_{22}-\Delta _{13}\big) \big\}$. So we get  $\mathcal{G}%
_{05}^{-1}$ and  $\mathcal{G}_{06}^{-2}$  again.

\item $\theta =\eta _{3}$. Let $\phi $ be the following automorphism:%
\begin{equation*}
\phi =%
\begin{pmatrix}
-1 & 0 & 0 \\ 
\frac{\alpha _{1}}{2} & 0 & 1 \\ 
-\frac{\alpha _{2}}{2} & 1 & 0%
\end{pmatrix}%
.
\end{equation*}%
Then $\theta \ast \phi =\big( 0,\ \Delta _{12},\ -\Delta_{13}\big) $. Thus we obtain the algebra $\mathcal{G}_{05}^{-1}$.

\item $\theta =\eta _{4}$. Let $\phi $ be the following automorphism:%
\begin{equation*}
\phi =%
\begin{pmatrix}
1 & 0 & 0 \\ 
-\frac{\alpha _{3}}{\alpha _{2}} & 1 & 0 \\ 
-\frac{\alpha _{3}^{2}-\alpha _{1}\alpha _{2}}{2\alpha _{2}} & 0 & \alpha
_{2}%
\end{pmatrix}%
.
\end{equation*}%
Then $\theta \ast \phi =\big( 0,\ \Delta _{12},\ \Delta_{22}-\Delta _{13}\big)$. Therefore, we have the algebra   $\mathcal{G}_{06}^{-2}$.

\item $\theta =\eta _{5}$. Let $\phi $ be the following automorphism:%
\begin{equation*}
\phi =%
\begin{pmatrix}
-1 & 0 & 0 \\ 
-\frac{\alpha _{3}^{2}-\alpha _{1}\alpha _{2}}{2\alpha _{2}} & 0 & \alpha
_{2} \\ 
{\alpha^{-1}_{2}}{\alpha _{3}}   & 1 & 0%
\end{pmatrix}%
.
\end{equation*}%
Then $\theta \ast \phi =\big( 0,\ \Delta _{12},\ \Delta_{22}-\Delta _{13}\big) $. Therefore we have the algebra   $\mathcal{G}_{06}^{-2}$.
\end{itemize}
\item 
\underline{$\mathcal{G}^{-}=\mathcal{L}_{03}^{0}$.} Choose an
arbitrary element $\theta =\big( B_{1},B_{2},B_{3}\big) \in {\rm Z}_{\delta}^{2}\big( \mathcal{L}_{03}^{ 0},\mathcal{L}_{03}^{0}\big) $. Then $\theta \in \{\eta _{1},\ldots
,\eta _{5}\}$ where%
\begin{longtable}{lcl}
$\eta _{1} $&$=$&$\big( 0,\ \alpha _{1}\Delta _{11}+\Delta_{12}, \ \alpha _{2}\Delta _{33}\big) ,$ \\
$\eta _{2} $&$=$&$\big( 0,\ \alpha _{1}\Delta _{11}+\Delta_{12}, \ \alpha _{2}\Delta _{11}+{\alpha^{-1} _{2}}{\alpha _{3}^{2}}\Delta _{33}+\alpha _{3}\Delta _{13}\big) ,$ \\
$\eta _{3} $&$=$&$\big( -2\delta \Delta _{11},\ \alpha _{1}\Delta_{11}-\Delta _{12},\ \alpha _{2}\Delta _{11}+\alpha_{3}\Delta _{13}+{\alpha^{-1} _{2}}{\alpha _{3}\big( \alpha _{3}+2\delta \big) 
} \Delta _{33}\big) ,$ \\
$\eta _{4} $&$=$&$\big( -2\delta \Delta _{11}, \ \alpha _{1}\Delta
_{11}-\Delta _{12},\ -2\delta \Delta _{13}+\alpha _{2}\Delta _{33}\big) ,$ \\
$\eta _{5} $&$=$&$\big( -2\delta \Delta _{11}, \ \alpha _{1}\Delta
_{11}-\Delta _{12}, \ \alpha _{2}\Delta _{33}\big) ,$
\end{longtable}%
for some $\alpha _{1},\alpha _{2},\alpha _{3}\in \mathbb{C}$.

\begin{itemize}
\item $\theta =\eta _{1}$. Let $\phi=\varphi_1 $  if $\alpha _{2}=0,$ or 
$\phi=\varphi_2$   if $\alpha _{2}\neq 0$:%
\begin{equation*}
\varphi_1=\begin{pmatrix}
1 & 0 & 0 \\ 
-\frac{\alpha _{1}}{2} & 1 & 0 \\ 
0 & 0 & 1%
\end{pmatrix}%
, \ 
\varphi_2=\begin{pmatrix}
1 & 0 & 0 \\ 
-\frac{\alpha _{1}}{2} & 1 & 0 \\ 
0 & 0 & \frac{1}{\alpha _{2}}%
\end{pmatrix}%
.
\end{equation*}%
Then $\theta \ast \phi \in \big\{\big( 0, \ \Delta _{12}, \ 0\big),\ \big(
0, \ \Delta _{12},\ \Delta _{33}\big) \big\}$. So we get the
algebras $\mathcal{G}_{05}^{0}$ and $\mathcal{G}_{07}$.

\item $\theta =\eta _{2}$. Let $\phi=\varphi_1$   if $\alpha _{3}=0,$ or 
$\phi=\varphi_2$  if $\alpha _{3}\neq 0$:%
\begin{equation*}
\varphi_1=\begin{pmatrix}
1 & 0 & 0 \\ 
-\frac{\alpha _{1}}{2} & 1 & 0 \\ 
0 & 0 & \alpha _{2}%
\end{pmatrix}%
, \ 
\varphi_2=\begin{pmatrix}
1 & 0 & 0 \\ 
-\frac{\alpha _{1}}{2} & 1 & 0 \\ 
-\frac{\alpha _{2}}{\alpha _{3}} & 0 & \frac{\alpha _{2}}{\alpha _{3}^{2}}%
\end{pmatrix}%
.
\end{equation*}%
Then $\theta \ast \phi \in \big\{\big( 0,\ \Delta _{12},\ \Delta _{11}\big), \ \big( 0,\ \Delta_{12},\ \Delta _{33}\big) \big\}$%
. So we get the algebras $\mathcal{G}_{08}$ and $\mathcal{G}_{07}$.

\item $\theta =\eta _{3}$. Let $\phi=\varphi_1$  if $\alpha _{3}=0$, or 
$\phi=\varphi_2$   if $\alpha _{3}=-2\delta $, or 
$\phi=\varphi_3$  if $\alpha _{3}\big( \alpha _{3}+2\delta \big) \neq 0$:%
\begin{equation*}
\varphi_1=\begin{pmatrix}
1 & 0 & 0 \\ 
\frac{\alpha _{1}}{2-2\delta } & 1 & 0 \\ 
-\frac{\alpha _{2}}{2\delta } & 0 & 1%
\end{pmatrix}%
, \ 
\varphi_2=\begin{pmatrix}
1 & 0 & 0 \\ 
\frac{\alpha _{1}}{2-2\delta } & 1 & 0 \\ 
\frac{\alpha _{2}}{2\delta } & 0 & 1%
\end{pmatrix}%
, \ 
\varphi_3=\begin{pmatrix}
1 & 0 & 0 \\ 
\frac{\alpha _{1}}{2-2\delta } & 1 & 0 \\ 
-\frac{\alpha _{2}}{2\delta +\alpha _{3}} & 0 & \frac{\alpha _{2}}{\alpha
_{3}\big( 2\delta +\alpha _{3}\big) }%
\end{pmatrix}%
.
\end{equation*}%
Then 
\begin{center}
    $\theta \ast \phi \in \big \{
\big( -2\delta \Delta_{11},\ -\Delta _{12}, \ 0\big) ,\ 
\big( -2\delta \Delta_{11},\ -\Delta _{12},\ -2\delta \Delta _{13}\big) ,\ 
\big(-2\delta \Delta _{11},\ -\Delta _{12},\ \Delta
_{33}\big) \big\}$.
\end{center} So we get the algebras $\mathcal{G}_{09}^{\delta},$ \
$\mathcal{G}_{10}^{\delta }$, and $\mathcal{G}_{11}^{\delta }$.

\item $\theta =\eta _{4}$. Let $\phi=\varphi_1$  if $\alpha _{2}=0,$ or 
$\phi=\varphi_2$  if $\alpha _{2}\neq 0$:%
\begin{equation*}
\varphi_1=\begin{pmatrix}
1 & 0 & 0 \\ 
\frac{\alpha _{1}}{2-2\delta } & 1 & 0 \\ 
0 & 0 & 1%
\end{pmatrix}%
, \
\varphi_2=\begin{pmatrix}
1 & 0 & 0 \\ 
\frac{\alpha _{1}}{2-2\delta } & 1 & 0 \\ 
\frac{2\delta }{\alpha _{2}} & 0 & \frac{1}{\alpha _{2}}%
\end{pmatrix}%
.
\end{equation*}%
Then $\theta \ast \phi \in 
\big\{\big( -2\delta \Delta_{11}, \ -\Delta _{12}, \ -2\delta \Delta _{13}\big) ,\ 
\big(-2\delta \Delta _{11}, \ -\Delta _{12},\ \Delta_{33}\big) \big\}$. 
So we get the algebras $\mathcal{G}_{10}^{\delta }$ and $%
\mathcal{G}_{11}^{\delta }$.

\item $\theta =\eta _{5}$. Let $\phi=\varphi_1$  if $\alpha _{2}=0,$ or 
$\phi=\varphi_2$  if $\alpha _{2}\neq 0$:%
\begin{equation*}
\varphi_1=\begin{pmatrix}
1 & 0 & 0 \\ 
\frac{\alpha _{1}}{2-2\delta } & 1 & 0 \\ 
0 & 0 & 1%
\end{pmatrix}%
, \
\varphi_2=\begin{pmatrix}
1 & 0 & 0 \\ 
\frac{\alpha _{1}}{2-2\delta } & 1 & 0 \\ 
0 & 0 & \frac{1}{\alpha _{2}}%
\end{pmatrix}%
.
\end{equation*}%
Then $\theta \ast \phi \in \big\{\big( -2\delta \Delta_{11},\ -\Delta _{12},\ 0\big),\ 
\big( -2\delta \Delta_{11},\ -\Delta _{12},\ \Delta _{33}\big) \big\}$. 
We get   $\mathcal{G}_{09}^{\delta }$ and $\mathcal{G}_{11}^{\delta }$.
\end{itemize}
\item 
\underline{$\mathcal{G}^{-}=\mathcal{L}_{03}^{1}$.} Choose an
arbitrary element $\theta =\big( B_{1},B_{2},B_{3}\big) \in {\rm Z}_{\delta
}^{2}\big( \mathcal{L}_{03}^{1},\mathcal{L}_{03}^{1}\big) $. Then $\theta \in \left\{ \eta _{1},\eta _{2}\right\},$ where%
\begin{eqnarray*}
\eta _{1} &=&\big( -2\delta \Delta _{11},\ 
\alpha _{1}\Delta_{11}-\Delta _{12},\ 
\alpha _{2}\Delta _{11}-\Delta_{13}\big) , \\
\eta _{2} &=&\big( 0,\ \alpha _{1}\Delta _{11}+\Delta_{12},\ \alpha _{2}\Delta _{11}+\Delta _{13}\big) ,
\end{eqnarray*}%
for some $\alpha _{1},\alpha _{2}\in \mathbb{C}$.

\begin{itemize}
\item $\theta =\eta _{1}$. Let $\phi $ be the following automorphism:%
\begin{equation*}
\phi =%
\begin{pmatrix}
1 & 0 & 0 \\ 
\frac{\alpha _{1}}{2-2\delta} & 1 & 0 \\ 
\frac{\alpha _{2}}{2-2\delta} & 0 & 1%
\end{pmatrix}%
.
\end{equation*}%
Then $\theta \ast \phi =\big( -2\delta \Delta _{11},\ -\Delta_{12},\ -\Delta _{13}\big) $. Thus we obtain the algebra $\mathcal{G%
}_{12}^{\delta }$.

\item $\theta =\eta _{2}$. Let $\phi $ be the following automorphism:%
\begin{equation*}
\phi =\allowbreak 
\begin{pmatrix}
1 & 0 & 0 \\ 
-\frac{\alpha _{1}}{2} & 1 & 0 \\ 
-\frac{\alpha _{2}}{2} & 0 & 1%
\end{pmatrix}%
.
\end{equation*}%
Then $\theta \ast \phi =\big( 0,\ \Delta _{12},\ \Delta_{13}\big) $. Thus we obtain the algebra $\mathcal{G}_{05}^{ 1}$.
\end{itemize}

\item 
\underline{$\mathcal{G}^{-}=\mathcal{L}_{03}^{\alpha \notin \big\{0,\pm 1 \big\}}$.}
Suppose first that $\alpha \in \big\{\delta +1,\ \frac{1}{\delta +1}\big\}$. Since  $%
\mathcal{L}_{03}^{\delta +1}\cong \mathcal{L}_{03}^{\frac{1}{%
\delta +1}}$, we may consider the algebra $\mathcal{L}_{03}^{\delta
+1}$. Choose an arbitrary element $\theta =\big( B_{1},B_{2},B_{3}\big)
\in {\rm Z}_{\delta }^{2}\big( \mathcal{L}_{03}^{\delta +1},\mathcal{L}%
_{03}^{ \delta +1}\big) $. Then $\theta \in \left\{ \eta _{1},\eta
_{2},\eta _{3}\right\} $ where%
\begin{longtable}{lcl}
$\eta _{1} $&$=$&$\big( 0,\ {2{\alpha^{-1}_{2}}\alpha _{3}}\Delta_{11}+\Delta _{12},\ \alpha _{1}\Delta _{11}+\alpha_{2}\Delta _{22}+\alpha _{3}\Delta _{12}+\big( \delta
+1\big) \Delta _{13}\big) ,$ \\
$\eta _{2} $&$=$&$\big( 0,\ \Delta _{12},\ \alpha _{1}\Delta_{11}+\alpha _{2}\Delta _{22}+\big( \delta +1\big) \Delta
_{13}\big) ,$ \\
$\eta _{3} $&$=$&$\big( 0, \ \alpha _{1}\Delta _{11}+\Delta_{12},\ \alpha _{2}\Delta _{11}+\big( \delta +1\big) \Delta_{13}\big) ,$
\end{longtable}%
for some $\alpha _{1},\alpha _{2},\alpha _{3}\in \mathbb{C}$.

\begin{itemize}
\item $\theta =\eta _{1}$. Let $\phi $ be the following automorphism:%
\begin{equation*}
\phi =%
\begin{pmatrix}
1 & 0 & 0 \\ 
-{\alpha^{-1} _{2}}{\alpha _{3}}  & 1 & 0 \\ 
\frac{\alpha _{3}^{2}-\alpha _{1}\alpha _{2}}{2\alpha _{2}(1+\delta)} & 0 & \alpha _{2}%
\end{pmatrix}%
.
\end{equation*}%
Then $\theta \ast \phi =\big( 0,\ \Delta _{12},\ \Delta
_{22}+\big( \delta +1\big) \Delta _{13}\big) $. So we get the
algebra $\mathcal{G}_{06}^{\delta \neq -2}$.

\item $\theta =\eta _{2}$. Let $\phi=\varphi_1$  if $\alpha _{2}=0,$ or
 $\phi=\varphi_2$  if $\alpha _{2}\neq 0$:%
\begin{equation*}
\varphi_1=\begin{pmatrix}
1 & 0 & 0 \\ 
0 & 1 & 0 \\ 
-\frac{\alpha _{1}}{2(\delta +1)} & 0 & 1%
\end{pmatrix}%
, \ 
\varphi_2=\begin{pmatrix}
1 & 0 & 0 \\ 
0 & 1 & 0 \\ 
-\frac{\alpha _{1}}{2(\delta +1)} & 0 & \alpha _{2}%
\end{pmatrix}%
.
\end{equation*}%
Then \begin{center}$\theta \ast \phi =\big \{\big( 0, \ \Delta _{12},\ \big( \delta+1\big) \Delta _{13}\big) , \ \big( 0,\ \Delta_{12},\ \Delta _{22}+\big( \delta +1\big) \Delta
_{13}\big) \big\}$.
 \end{center} So we get the algebras $\mathcal{G}_{05}^{ \delta
+1}$ and $\mathcal{G}_{06}^{\delta \neq -2}$.

\item $\theta =\eta _{3}$. Let $\phi $ be the following automorphism:%
\begin{equation*}
\phi =%
\begin{pmatrix}
1 & 0 & 0 \\ 
-\frac{\alpha _{1}}{2} & 1 & 0 \\ 
-\frac{\alpha _{2}}{2(\delta +1)} & 0 & 1%
\end{pmatrix}%
.
\end{equation*}%
Then $\theta \ast \phi =\big( 0,\ \Delta _{12},\ \big( \delta
+1\big) \Delta _{13}\big) $. So we get the algebra $\mathcal{G}%
_{05}^{\delta +1}$.
\end{itemize}

\item 
\underline{$\mathcal{G}^{-}=\mathcal{L}_{03}^{\alpha \notin \big\{0,\pm 1 \big\}}$.} Assume now that $\alpha \notin \big\{\delta +1,\frac{1}{\delta +1}\big\}$. Choose an
arbitrary element ${\theta =\big( B_{1},B_{2},B_{3}\big) \in {\rm Z}_{\delta
}^{2}\big( \mathcal{L}_{03}^{\alpha \notin \{0,\pm 1\}},\mathcal{L}_{03}^{\alpha
 \notin \{0,\pm 1\}}\big)}$. \ Then 
     $\theta =\big( 0, \ \alpha _{1}\Delta
_{11}+\Delta _{12},\ \alpha _{2}\Delta _{11}+\alpha
\Delta _{13}\big) $ for some $\alpha_1, \alpha_2 \in {\mathbb C}.$ 
Let $\phi $ be the following
automorphism: 
\begin{equation*}
\phi =%
\begin{pmatrix}
1 & 0 & 0 \\ 
-\frac{\alpha _{1}}{2} & 1 & 0 \\ 
-\frac{\alpha _{2}}{2\alpha } & 0 & 1%
\end{pmatrix}%
.
\end{equation*}%
Then $\theta \ast \phi =\big( 0,\ \Delta _{12},\ \alpha \Delta
_{13}\big) $. So we get the algebras $\mathcal{G}_{05}^{\alpha }$.
\end{enumerate}
\end{proof}

\begin{corollary}
Let $\mathcal{G}$ be a nontrivial complex $3$-dimensional $\cap$-Novikov algebra (i.e., it is a $\delta$-Novikov algebra for all possible values $\delta$). 
Then $\mathcal{G}$ is an associative commutative
algebra listed in {\rm Proposition \ref{3-dim assoc}} or isomorphic to one of the
following algebras:

\begin{longtable}{lllllll}
$\mathcal{G}_{01}$ & $:$ & $e_{1}\cdot e_{2}=e_{3}$ & $e_{2}\cdot e_{1}=-e_{3}$ \\

$\mathcal{G}_{02}^{\alpha}$ & $:$ & $e_{1}\cdot e_{1}=e_{3}$ & $e_{1}\cdot e_{2}=e_{3}$ & $e_{2}\cdot e_{1}=-e_{3}$ & $e_{2}\cdot e_{2}=\alpha e_{3}$ \\

$\mathcal{G}_{03}^{\alpha}$ & $:$ & $e_{1}\cdot e_{1}=e_{2}$ & $e_{1}\cdot e_{2}=\big( \alpha +1\big) e_{3}$ & $e_{2}\cdot e_{1}=\big( \alpha -1\big) e_{3}$ \\

$\mathcal{G}_{04}$ & $:$ & $e_{1}\cdot e_{2}=e_{2}$ & $e_{1}\cdot e_{3}=e_{2}+e_{3}$ \\

$\mathcal{G}_{05}^{\alpha}$ & $:$ & $e_{1}\cdot e_{2}=e_{2}$ & $e_{1}\cdot e_{3}=\alpha e_{3}$ \\

$\mathcal{G}_{07}$ & $:$ & $e_{1}\cdot e_{2}=e_{2}$ & $e_{3}\cdot e_{3}=e_{3}$ \\

$\mathcal{G}_{08}$ & $:$ & $e_{1}\cdot e_{1}=e_{3}$ & $e_{1}\cdot e_{2}=2e_{2}$ \\
\end{longtable}
\noindent
All listed algebras are non-isomorphic except: $\mathcal{G}_{05}^{\alpha
}\cong \mathcal{G}_{05}^{\alpha ^{-1}}$.
\end{corollary}

\newpage

\section{The geometric classification of 
  algebras}

\subsection{Preliminaries: definitions and notation}
Given an $n$-dimensional vector space $\mathbb V$, the set ${\rm Hom}(\mathbb V \otimes \mathbb V,\mathbb V) \cong \mathbb V^* \otimes \mathbb V^* \otimes \mathbb V$ is a vector space of dimension $n^3$. This space has the structure of the affine variety $\mathbb{C}^{n^3}.$ Indeed, let us fix a basis $e_1,\dots,e_n$ of $\mathbb V$. Then any $\mu\in {\rm Hom}(\mathbb V \otimes \mathbb V,\mathbb V)$ is determined by $n^3$ structure constants $c_{ij}^k\in\mathbb{C}$ such that
$\mu(e_i\otimes e_j)=\sum\limits_{k=1}^nc_{ij}^ke_k$. A subset of ${\rm Hom}(\mathbb V \otimes \mathbb V,\mathbb V)$ is {\it Zariski-closed} if it can be defined by a set of polynomial equations in the variables $c_{ij}^k$ ($1\le i,j,k\le n$).

Let $T$ be a set of polynomial identities.
The set of algebra structures on $\mathbb V$ satisfying polynomial identities from $T$ forms a Zariski-closed subset of the variety ${\rm Hom}(\mathbb V \otimes \mathbb V,\mathbb V)$. We denote this subset by $\mathbb{L}(T)$.
The general linear group ${\rm GL}(\mathbb V)$ acts on $\mathbb{L}(T)$ by conjugations:
$$ (g * \mu )(x\otimes y) = g\mu(g^{-1}x\otimes g^{-1}y)$$
for $x,y\in \mathbb V$, $\mu\in \mathbb{L}(T)\subset {\rm Hom}(\mathbb V \otimes\mathbb V, \mathbb V)$ and $g\in {\rm GL}(\mathbb V)$.
Thus, $\mathbb{L}(T)$ is decomposed into ${\rm GL}(\mathbb V)$-orbits that correspond to the isomorphism classes of algebras.
Let ${\mathcal O}(\mu)$ denote the orbit of $\mu\in\mathbb{L}(T)$ under the action of ${\rm GL}(\mathbb V)$ and $\overline{{\mathcal O}(\mu)}$ denote the Zariski closure of ${\mathcal O}(\mu)$.

Let $\bf A$ and $\bf B$ be two $n$-dimensional algebras satisfying the identities from $T$, and let $\mu,\lambda \in \mathbb{L}(T)$ represent $\bf A$ and $\bf B$, respectively.
We say that $\bf A$ degenerates to $\bf B$ and write $\bf A\to \bf B$ if $\lambda\in\overline{{\mathcal O}(\mu)}$.
Note that in this case we have $\overline{{\mathcal O}(\lambda)}\subset\overline{{\mathcal O}(\mu)}$. Hence, the definition of degeneration does not depend on the choice of $\mu$ and $\lambda$. If $\bf A\not\cong \bf B$, then the assertion $\bf A\to \bf B$ is called a {\it proper degeneration}. We write $\bf A\not\to \bf B$ if $\lambda\not\in\overline{{\mathcal O}(\mu)}$.

Let $\bf A$ be represented by $\mu\in\mathbb{L}(T)$. Then  $\bf A$ is  {\it rigid} in $\mathbb{L}(T)$ if ${\mathcal O}(\mu)$ is an open subset of $\mathbb{L}(T)$.
 Recall that a subset of a variety is called irreducible if it cannot be represented as a union of two non-trivial closed subsets.
 A maximal irreducible closed subset of a variety is called an {\it irreducible component}.
It is well known that any affine variety can be represented as a finite union of its irreducible components in a unique way.
The algebra $\bf A$ is rigid in $\mathbb{L}(T)$ if and only if $\overline{{\mathcal O}(\mu)}$ is an irreducible component of $\mathbb{L}(T)$.

\medskip

\noindent {\bf Method of the description of degenerations of algebras.} In the present work we use the methods applied to Lie algebras in \cite{GRH}.
First of all, if $\bf A\to \bf B$ and $\bf A\not\cong \bf B$, then $\mathfrak{Der}(\bf A)<\mathfrak{Der}(\bf B)$, where $\mathfrak{Der}(\bf A)$ is the   algebra of derivations of $\bf A$. We compute the dimensions of algebras of derivations and check the assertion $\bf A\to \bf B$ only for such $\bf A$ and $\bf B$ that $\mathfrak{Der}(\bf A)<\mathfrak{Der}(\bf B)$.

To prove degenerations, we construct families of matrices parametrized by $t$. Namely, let $\bf A$ and $\bf B$ be two algebras represented by the structures $\mu$ and $\lambda$ from $\mathbb{L}(T)$ respectively. Let $e_1,\dots, e_n$ be a basis of $\mathbb  V$ and $c_{ij}^k$ ($1\le i,j,k\le n$) be the structure constants of $\lambda$ in this basis. If there exist $a_i^j(t)\in\mathbb{C}$ ($1\le i,j\le n$, $t\in\mathbb{C}^*$) such that $E_i^t=\sum\limits_{j=1}^na_i^j(t)e_j$ ($1\le i\le n$) form a basis of $\mathbb V$ for any $t\in\mathbb{C}^*$, and the structure constants of $\mu$ in the basis $E_1^t,\dots, E_n^t$ are such rational functions $c_{ij}^k(t)\in\mathbb{C}[t]$ that $c_{ij}^k(0)=c_{ij}^k$, then $\bf A\to \bf B$.
In this case  $E_1^t,\dots, E_n^t$ is called a {\it parametrized basis} for $\bf A\to \bf B$.
In  case of  $E_1^t, E_2^t, \ldots, E_n^t$ is a {\it parametric basis} for ${\rm A}\to {\bf B},$ it will be denoted by
${\rm A}\xrightarrow{(E_1^t, E_2^t, \ldots, E_n^t)} {\bf B}$. 
To simplify our equations, we will use the notation $A_i=\langle e_i,\dots,e_n\rangle,\ i=1,\ldots,n$ and write simply $A_pA_q\subset A_r$ instead of $c_{ij}^k=0$ ($i\geq p$, $j\geq q$, $k< r$).


Let ${\rm A}(*):=\{ {\rm A}(\alpha)\}_{\alpha\in I}$ be a series of algebras, and let $\bf B$ be another algebra. Suppose that for $\alpha\in I$, $\bf A(\alpha)$ is represented by the structure $\mu(\alpha)\in\mathbb{L}(T)$ and $\bf B$ is represented by the structure $\lambda\in\mathbb{L}(T)$. Then we say that $\bf A(*)\to \bf B$ if $\lambda\in\overline{\{{\mathcal O}(\mu(\alpha))\}_{\alpha\in I}}$, and $\bf A(*)\not\to \bf B$ if $\lambda\not\in\overline{\{{\mathcal O}(\mu(\alpha))\}_{\alpha\in I}}$.

Let $\bf A(*)$, $\bf B$, $\mu(\alpha)$ ($\alpha\in I$) and $\lambda$ be as above. To prove $\bf A(*)\to \bf B$ it is enough to construct a family of pairs $(f(t), g(t))$ parametrized by $t\in\mathbb{C}^*$, where $f(t)\in I$ and $g(t)\in {\rm GL}(\mathbb V)$. Namely, let $e_1,\dots, e_n$ be a basis of $\mathbb V$ and $c_{ij}^k$ ($1\le i,j,k\le n$) be the structure constants of $\lambda$ in this basis. If we construct $a_i^j:\mathbb{C}^*\to \mathbb{C}$ ($1\le i,j\le n$) and $f: \mathbb{C}^* \to I$ such that $E_i^t=\sum\limits_{j=1}^na_i^j(t)e_j$ ($1\le i\le n$) form a basis of $\mathbb V$ for any  $t\in\mathbb{C}^*$, and the structure constants of $\mu({f(t)})$ in the basis $E_1^t,\dots, E_n^t$ are such rational functions $c_{ij}^k(t)\in\mathbb{C}[t]$ that $c_{ij}^k(0)=c_{ij}^k$, then $\bf A(*)\to \bf B$. In this case  $E_1^t,\dots, E_n^t$ and $f(t)$ are called a parametrized basis and a {\it parametrized index} for $\bf A(*)\to \bf B$, respectively.

We now explain how to prove $\bf A(*)\not\to\mathcal  \bf B$.
Note that if $\mathfrak{Der} ( \bf A(\alpha))  > \mathfrak{Der} (  \bf B)$ for all $\alpha\in I$ then $\bf A(*)\not\to\bf B$.
One can also use the following  Lemma, whose proof is the same as the proof of   \cite[Lemma 1.5]{GRH}.

\begin{lemma}\label{gmain}
Let $\mathfrak{B}$ be a Borel subgroup of ${\rm GL}(\mathbb V)$ and ${\rm R}\subset \mathbb{L}(T)$ be a $\mathfrak{B}$-stable closed subset.
If $\bf A(*) \to \bf B$ and for any $\alpha\in I$ the algebra $\bf A(\alpha)$ can be represented by a structure $\mu(\alpha)\in{\rm R}$, then there is $\lambda\in {\rm R}$ representing $\bf B$.
\end{lemma}

\subsection{The geometric classification of   
  $\cap$-Novikov algebras}

\begin{theoremG0}\label{thm:geo_cap}
The variety of complex $3$-dimensional $\cap$-Novikov algebras has dimension $9$ and it has $3$ irreducible components defined by  
\begin{center}
$\mathcal{C}_1=\overline{\mathcal{O}( {\rm A}_{07})},$ \
$\mathcal{C}_2=\overline{\mathcal{O}( \mathcal{G}_{05}^{\alpha})},$ \ and
$\mathcal{C}_3=\overline{\mathcal{O}( \mathcal{G}_{07})}.$ \
\end{center}
In particular, there are only $2$ rigid algebras in this variety.
\end{theoremG0}

\begin{proof}
After carefully  checking  the dimensions of orbit closures of the more important for us algebras, we have 

\begin{longtable}{rclrclrclrcl}
$\dim \mathcal{O}({\rm A}_{07})$&$=$&$9,$ &
$\dim \mathcal{O}(\mathcal{G}_{05}^{\alpha})$&$=$&$8,$ & 
$\dim \mathcal{O}(\mathcal{G}_{07})$&$=$&$ 8$. &
\end{longtable}
\noindent
Thanks to~\cite{MS}, ${\rm A}_{07}$ is rigid in the variety of associative commutative algebras and each commutative associative algebra is in the irreducible component defined by ${\rm A}_{07}$. Since ${\rm A}_{07}$ is commutative, we have
${\rm A}_{07} \not\to \big\{\mathcal{G}_{05}^{\alpha},\  \mathcal{G}_{07}  \big\}$.

All necessary degenerations are given below 
\begin{longtable}{|lcl|}
\hline
    $\mathcal{G}_{02}^{0} $&$\xrightarrow{ (te_1,\ e_2,\ te_3)} $&$\mathcal{G}_{01}$  \\
\hline
    $\mathcal{G}_{03}^{i\sqrt{\alpha}} $&$\xrightarrow{ (te_1-it^{-1}\alpha^{-1/2}e_2+i(t^4-1)t^{-3}\alpha^{-1/2}e_3,\ -i\sqrt{\alpha}t e_1+te_3,\  e_3)} $&$\mathcal{G}_{02}^{\alpha}$  \\
\hline

        $\mathcal{G}_{07} $&$\xrightarrow{ (2t e_1+e_2+(t+\alpha t) e_3,\ 2t e_2+(1-\alpha)^2 t^2e_3,\ -(1-\alpha)^2 t^3 )} $&$\mathcal{G}_{03}^\alpha$\\
\hline

$\mathcal{G}_{05}^{t+1} $&$\xrightarrow{ ( e_1,\ te_3,\ e_2 + e_3)} $&$\mathcal{G}_{04}$ \\
\hline    
    $\mathcal{G}_{05}^{t} $&$\xrightarrow{ (2e_1+e_3,\ te_2+t^2(t-1)^{-1}e_3,\ 2te_3)} $&$\mathcal{G}_{08}$\\
\hline

\end{longtable}

\end{proof}

\subsection{The geometric classification of   
  $\delta$-Novikov algebras  $\big(\delta \notin \big\{0,\pm 1\big\}\big)$}

\begin{theoremG1}\label{thm:geo_delt}
The variety of complex $3$-dimensional $\delta$-Novikov algebras $(\delta \notin \{0,\pm 1\})$ 
has dimension $9$ and it has $6$ irreducible components defined by  
\begin{center}
$\mathcal{C}_1=\overline{\mathcal{O}( {\rm A}_{07})},$ \
$\mathcal{C}_2=\overline{\mathcal{O}( \mathcal{G}_{05}^{\alpha})},$ \  
$\mathcal{C}_3=\overline{\mathcal{O}( \mathcal{G}_{06}^{\delta})},$ \  
$\mathcal{C}_4=\overline{\mathcal{O}( \mathcal{G}_{07})},$ \  
$\mathcal{C}_5=\overline{\mathcal{O}( \mathcal{G}_{11}^{\delta})},$ \ and
$\mathcal{C}_6=\overline{\mathcal{O}( \mathcal{G}_{12}^{\delta})}.$ \
\end{center}
In particular, there are only $5$ rigid algebras in this variety.
\end{theoremG1}

\begin{proof}
After carefully  checking  the dimensions of orbit closures of the more important for us algebras, we have 

\begin{longtable}{rclrclrclrcl}
$\dim \mathcal{O}({\rm A}_{07})$&$=$&$9,$ &
$\dim \mathcal{O}(\mathcal{G}_{05}^{\alpha})$&$=$&$8,$ & 
$\dim \mathcal{O}(\mathcal{G}_{07})$&$=$&$ 8,$ \\
$\dim \mathcal{O}(\mathcal{G}_{06}^{\delta})$&$=$&$8,$ & 
$\dim \mathcal{O}(\mathcal{G}_{11}^{\delta})$&$=$&$8,$ & 
$\dim \mathcal{O}(\mathcal{G}_{12}^{\delta})$&$=$&$5.$ & 

\end{longtable}
\noindent
Thanks to~\cite{MS}, ${\rm A}_{07}$ is rigid in the variety of associative commutative algebras and each commutative associative algebra is in the irreducible component defined by ${\rm A}_{07}$. Since ${\rm A}_{07}$ is commutative, we have
${\rm A}_{07} \not\to \big\{\mathcal{G}_{05}^{\alpha},\  \mathcal{G}_{06}^{\delta}, \ \mathcal{G}_{07}, \ \mathcal{G}_{11}^{\delta}, \ 
\mathcal{G}_{12}^{\delta}\big\}$.

All necessary degenerations and non-degenerations are given in Theorem G0 and below 

 \begin{longtable}{rcl}
    $\mathcal{G}_{11}^{\delta} \ \xrightarrow{ (e_1,\ e_2,\ te_3)}  \mathcal{G}_{09}^\delta$ &  \mbox{and}  &
    $\mathcal{G}_{11}^{\delta} \ \xrightarrow{ (e_1 +\delta e_3,\ e_2,\ te_3)}\ \mathcal{G}_{10}^\delta;$ \\
$\mathcal{G}_{11}^{\delta} \not\to \mathcal{G}_{12}^{\delta}$ due to ${\mathcal R}  = \left\{  A_2^2 \subseteq A_3, c_{13}^3=c_{31}^3 \right\}$ & \mbox{and}  &$\mathcal{G}_{06}^{\delta} \not\to \mathcal{G}_{12}^{\delta}$ due to ${\mathcal R}  = \left\{  A_1^2 \subseteq A_2  \right\}.$ 

 \end{longtable}

\end{proof}

\subsection{The geometric classification of   
  anti-Novikov algebras}

\begin{theoremG2}\label{thm:geo_delt}
The variety of complex $3$-dimensional anti-Novikov algebras  
has dimension $9$ and it has $4$ irreducible components defined by  
\begin{center}
$\mathcal{C}_1=\overline{\mathcal{O}( {\rm A}_{07})},$ \
$\mathcal{C}_2=\overline{\mathcal{O}( \rm{N}_{13} )},$ \
$\mathcal{C}_3=\overline{\mathcal{O}( \mathcal{G}_{05}^{\alpha})},$ \  and 
$\mathcal{C}_4=\overline{\mathcal{O}( \mathcal{G}_{12}^{-1})}.$ \

\end{center}
In particular, there are only $3$ rigid algebras in this variety.
\end{theoremG2}

\begin{proof}
After carefully  checking  the dimensions of orbit closures of the more important for us algebras, we have 

\begin{longtable}{rclrclrclrcl}
$\dim \mathcal{O}({\rm A}_{07})$&$=$&$9,$ &
$\dim \mathcal{O}(\rm{N}_{13})$&$=$&$9,$ &
$\dim \mathcal{O}(\mathcal{G}_{05}^{\alpha})$&$=$&$8,$ & 
$\dim \mathcal{O}(\mathcal{G}_{12}^{-1})$&$=$&$5.$  

\end{longtable}
\noindent
Thanks to~\cite{MS}, ${\rm A}_{07}$ is rigid in the variety of associative commutative algebras and each commutative associative algebra is in the irreducible component defined by ${\rm A}_{07}$. Since ${\rm A}_{07}$ is commutative, we have
${\rm A}_{07} \not\to \big\{\mathcal{G}_{05}^{\alpha},\  \mathcal{G}_{12}^{-1}\big\}$.

All necessary degenerations   are given in Theorems G0 and G1, and below 

 \begin{longtable}{|lcr|lcr|}

\hline
     ${\rm N}_{13}  $&$ \xrightarrow{ ( \frac{1}{2} e_1 -\frac{1}{2}t^6 e_3,\ t e_2,\ t^8 e_3)} $&$  {\mathcal G}_{06}^{-1}$ &
     ${\rm N}_{13}  $&$ \xrightarrow{ (\frac{1}{2}e_1,\ t e_2,\  e_3)}$&$  {\mathcal G}_{07}$    

\\
\hline
     
   ${\rm N}_{13} $&$ \xrightarrow{ (-\frac{1}{2}e_1+\frac{1}{\sqrt{2}}e_2,\ t e_2,\ e_3)} $&$ {\mathcal G}_{11}^{-1}$    &
    ${\rm N}_{13}  $&$ \xrightarrow{ (e_1+\sqrt{2} e_2+2 e_3,\ \frac{t(t-1)}{2\sqrt{2}}e_2+\frac{t^2}{2}e_3,\ te_3)} $&$ {\rm N}_{09}$    \\
    \hline
    ${\rm N}_{13}  $&$ \xrightarrow{ (e_1,\ e_2,\ te_3)} $&$ {\rm N}_{12}$       &
    ${\rm  N}_{13}  $&$ \xrightarrow{ (e_1 - t^2 e_3,\ t e_2,\ t^4e_3)} $&$ {\rm  N}_{15}$    \\
\hline

 \end{longtable}

$\rm{N}_{13} \not\to \big\{\mathcal{G}_{05}^\alpha, \mathcal{G}_{12}^{-1}\big\}$ due to 
${\mathcal R}  = \left\{  A_1A_3 + A_3A_1 \subseteq A_3, \ 
c_{12}^1=c_{21}^1, \ 
c_{13}^3=c_{31}^3, \ 
c_{23}^3=c_{32}^3 \right\}.$  

\end{proof}

\end{document}